\documentclass[11pt]{amsart}
\usepackage{amssymb,comment,amsmath,wasysym}
\usepackage{bbm}
\usepackage{graphicx} 
\usepackage[english]{babel} 
\usepackage{float}
\usepackage{mathtools}
\usepackage{hyperref}
\usepackage[noadjust]{cite}
\usepackage[margin=3.2cm]{geometry}
\hypersetup{ 
    colorlinks=true,       % false: boxed links; true: colored links
    linkcolor=blue,          % color of internal links
    citecolor=blue,        % color of links to bibliography
    filecolor=blue,      % color of file links
    urlcolor=blue           % color of external links
}

\usepackage{array}

\usepackage{pstricks} % Compile using XeLatex, View PDF
\usepackage{pst-all}
\usepackage{epsfig}
\usepackage{pst-grad} % For gradients
\usepackage{pst-plot} % For axes
\usepackage[space]{grffile} % For spaces in paths
\usepackage{etoolbox} % For spaces in paths
\usepackage{color} %for color text
\makeatletter % For spaces in paths
\patchcmd\Gread@eps{\@inputcheck#1 }{\@inputcheck"#1"\relax}{}{}
\makeatother

\usepackage[all]{xy}
\usepackage{mathrsfs}
\usepackage{graphics}
\usepackage{amsthm}

\theoremstyle{plain}\newtheorem{theorem}{Theorem}[section]\newtheorem{Theorem}{Theorem}\newtheorem{Corollary}{Corollary}\newtheorem{proposition}[theorem]{Proposition}\newtheorem{lemma}[theorem]{Lemma}

\theoremstyle{definition}\newtheorem{notation}[theorem]{Notation}\newtheorem{remark}[theorem]{Remark}%\newtheorem{construction}[subsection]{Construction}

\def\Q{\mathbb{Q}}\def\C{\mathbb{C}}\def\N{\mathbb{N}}\def\Z{\mathbb{Z}}\def\R{\mathbb{R}}\def\ZZ2{\mathbb{\Z/ 2\Z}}
\def\sb{\subset}\def\lb{\langle}\def\rb{\rangle}\def\ot{\otimes}\def\t{\times}\def\sm{\setminus}
\def\c{\gamma}\def\a{\alpha}\def\b{\beta}\def\d{\delta}\def\e{\epsilon}\def\De{\Delta}\def\La{\Lambda}\def\o{\omega}\def\la{\lambda}\def\la{\lambda}\def\p{\partial}\def\S{\Sigma}
\def\ov{\overline}\def\wt{\widetilde}
\def\sl2{\mathfrak{sl}_2}
\def\su2{\mathfrak{su}(2)}
\def\tr{\text{tr}\,}
\def\id{\text{id}}\def\End{\text{End}}
\def\Hom{\text{Hom}\,}
\def\deg{\text{deg}}

\def\Uq{U_q(\mf{sl}_2)}\def\Uh{U_h(\sl2)}\def\Usl2{U_q(\sl2)}\def\usl2{\wt{U}_q(\sl2)}\def\Uq{U_q(\mf{sl}_2)}\def\Uh{U_h(\sl2)}\def\uq{\mathfrak{u}_q(\sl2)}

\def\Rep{\text{Rep}}

\def\tt{\theta}
\def\unit{\mathbbm{1}}

\def\CC{\mathcal{C}}

\def\ZZ{\mathcal{Z}}

\def\CC{\mathcal{C}}

\def\HH{H=\{\Ha\}_{\a\in G}}

\def\Z{\mathbb{Z}}\def\gg{\mathfrak{g}}
\def\C{\mathbb{C}}\def\N{\mathbb{N}}
\def\ot{\otimes}\def\id{\text{id}}

\def\Uqg{U_q(\gg)}
\def\sl{\mathfrak{sl}}

\def\TT{\overline{\Theta}}
\def\UHr{\UHq}\def\qa{q}
\def\CC{\mathcal{C}}

\def\rev{\overrightarrow{\text{ev}}}

\def\rcoev{\overrightarrow{\text{coev}}}
\def\lev{\overleftarrow{\text{ev}}}

\def\lcoev{\overleftarrow{\text{coev}}}

\def\sl{\mathfrak{sl}_2}

\def\uq{\mathfrak{u}_q(\sl)}
\def\HH{\mathcal{H}}
\def\ADO{ADO_p}
\def\top{\mathrm{top}}

\def\KK{\mathbb{K}}

\def\Uqg{U_q(\mathfrak{gl}(2|1))}

\def\sl{\mathfrak{sl}}

\def\uq{\overline{U}_q(\mathfrak{sl}_2)}

\def\gg{\mathfrak{g}}

\def\ZZ{\mathcal{Z}}

\def\tt{\theta}

\def\uq{\mathfrak{u}_{\zeta}(\sl_2)}

\def\gg{\mathfrak{g}}

\def\qi{\qi}

\def\hh{\mathfrak{h}}

\def\Ur{\ov{U}_q(\gg)}
\def\UHqg{U^H_{q}(\gg)}\def\Uh{U(\mathfrak{h})}\def\Uqg{U_q(\gg)}
\def\UHr{\ov{U}^H_{q}(\gg)}\def\Ur{\ov{U}_{q}(\gg)}

\def\Uq{U_q}
\def\Uqg{U_q(\gg)}

\def\hh{\mathfrak{h}}\def\CC{\mathcal{C}}
\def\uqg{\mathfrak{u}_q(\gg)}

\def\uu{\mathfrak{u}}
\def\uqb{\uu_q(\mathfrak{b})}

\def\Dep{\De^+}

\def\ADO{\text{ADO}}
\def\ADOgr{\text{ADO}_{\gg,r}}
\def\topgr{\text{top}_{\gg,r}}
\def\Va{\Vla}
\def\Vad{\Va^*}

\def\ZZ{\mathcal{Z}}

\def\tt{\theta}

\def\uq{\mathfrak{u}_{\zeta}(\sl_2)}

\def\gg{\mathfrak{g}}

\def\qi{\qi}

\def\hh{\mathfrak{h}}

\def\Ur{\ov{U}_q(\gg)}
\def\UHqg{U^H_{q}(\gg)}\def\Uh{U(\mathfrak{h})}\def\Uqg{U_q(\gg)}
\def\UHr{\ov{U}^H_{q}(\gg)}\def\Ur{\ov{U}_{q}(\gg)}

\def\Uqg{U_q(\gg)}

\def\hh{\mathfrak{h}}\def\CC{\mathcal{C}^H}
\def\uqg{\mathfrak{u}_q(\gg)}

\def\uu{\mathfrak{u}}
\def\uqb{\uu_q(\mathfrak{b})}

\def\Dep{\De^+}

\def\ADO{\text{ADO}}
\def\ADOgr{\text{ADO}_{\gg,r}}

\def\rev{\overrightarrow{\text{ev}}}

\def\rcoev{\overrightarrow{\text{coev}}}
\def\lev{\overleftarrow{\text{ev}}}

\def\lcoev{\overleftarrow{\text{coev}}}

\def\coev{\text{coev}}

\def\XX{\mathbb{X}}

\def\rank{\text{rank}}
\def\tt{\boldsymbol{t}}
\def\leastgr{\text{least}_{\gg,r}}
\def\Va{\Vla}

\def\Im{I_{max}}
\def\XX{X^{\overline{\e}}}
\def\JJ{\mathcal{J}_X}

\def\FF{\overline{F}}
\def\cl{\text{cl}}

\def\HH{\mathcal{H}}\def\KK{\mathcal{K}}
\def\TT{\ov{\Theta}}
\def\Vla{\Vla}
\def\qa{q}

\def\rev{\overrightarrow{\text{ev}}}

\def\rcoev{\overrightarrow{\text{coev}}}
\def\lev{\overleftarrow{\text{ev}}}

\def\lcoev{\overleftarrow{\text{coev}}}

\def\coev{\text{coev}}

\def\XX{\mathbb{X}}
\def\UHrqg{\ov{U}^H_q(\gg)}

\def\CCgr{\CC_{\gg,r}}

\def\Imax{I_{max}}
\def\CCgr{\CC_{\gg,r}}

\def\UHqg{U^H_q(\gg)}\def\qi{q}

\def\TS{T_{\Sigma}}

\def\Imax{I_{max}}

\def\uqLW{\mathfrak{u}_q^{L_W}}
\def\uqL{\mathfrak{u}_q^L}
\def\uqLg{\mathfrak{u}_q^L(\gg)}
\def\Pgr{P_{\gg,r}}

\def\lla{\boldsymbol{\la}}
\def\Vla{V_{\lla}}\def\Va{V_{\lla}}

\title[Top coefficients of ADO invariants are Murasugi-sum multiplicative]{Top coefficients of ADO invariants are Murasugi-sum multiplicative}

\author[D. López Neumann]{Daniel López Neumann}
\address{Universidad de Concepción, Concepción, Chile}
\email{dlopezn@udec.cl}

\author[R. van der Veen]{Roland van der Veen}
\address{University of Groningen, Netherlands}
\email{r.i.van.der.veen@rug.nl}

\begin{document}

\maketitle

\begin{abstract}

The Akutsu-Deguchi-Ohtsuki (ADO) invariants of links are a non-semisimple version of the colored Jones polynomials. We prove that the top coefficients of the ADO invariants of links in $S^3$ are multiplicative under Murasugi-sum along connected minimal genus Seifert surfaces. More generally, for any simple complex Lie algebra $\gg$ of type ADE and any $r\geq 3$, we prove that the top coefficient of the $r$-th $\gg$-ADO invariant is multiplicative under Murasugi-sum.
As a corollary, we get that all the $\gg$-ADO invariants of fibred links are $q$-monic, generalizing a previous result of the authors for $\sl_2$. Moreover, whenever $\uqg$ is ribbon, we find an expression for the top coefficient of the $r$-th $\gg$-ADO invariant in terms of Hopf algebra integrals of universal $\uqg$-invariants of bottom tangles.

\end{abstract}

\section{Introduction}

Murasugi-sum is a way to glue two embedded surfaces in 3-space along an embedded polygon to get a third surface (see Figure \ref{fig:Murasugi sum}). This operation was introduced by Murasugi in \cite{Murasugi:genus-alternating1, Murasugi:genus-alternating2} to prove that the Alexander polynomial of an alternating link detects the Seifert genus\footnote{See \cite{Kindred:simple-proof-Murasugi-thm} for a nice exposition of the proof.}. A deep theorem of Gabai \cite{Gabai:Murasugi1, Gabai:Murasugi2} states that a Murasugi-sum of two surfaces has minimal genus for the link that appears as its boundary if and only if both surfaces have minimal genus (for their respective boundaries). Thus, given a link $L\sb S^3$, we say that {\em $L$ is a Murasugi-sum of links $L_1,L_2\sb S^3$} if $L$ bounds a Murasugi-sum of minimal genus Seifert surfaces of $L_1$ and $L_2$\footnote{Note that if we didn't add the condition of minimal genus to define Murasugi-sum of links, then the resulting operation would not be quite interesting by \cite{Able-Hirasawa:Murasugi} or \cite{Thompson:a-note-Murasugi}.}. This operation depends on various choices (the Seifert surfaces chosen and the polygons), hence several distinct links might be obtained by Murasugi-summing two given links in general. 
\medskip

In light of Gabai's theorem, whenever a property or invariant of links is related to the Seifert genus, it is natural to ask whether it behaves well under Murasugi-sum. This is the case for fibredness \cite{Gabai:fibred}, strong quasi-positivity \cite{Rudolph:quasipositive-plumbing}, the Alexander polynomial and knot Floer homology \cite{CHS:Murasugi-sum-extremal-Floer}, which all were first known to relate to the Seifert genus and later shown to behave well under Murasugi-sum.

\medskip

Recently, it has been found that the so called ``non-semisimple'' quantum link polynomials have genus bounds \cite{LNV:genus, LNV:genus-unrolled, KT:Links-Gould, GLN:genus-for-Lie-superalgebras}. This term refers to Reshetikhin-Turaev link invariants \cite{RT1} in which the underlying braided tensor category is non-semisimple, the prototype being categories of modules over (unrolled) quantum groups at roots of unity as studied by Geer-Patureau and collaborators \cite{GP:trace-projective-quantum-groups, GP:nonrestricted, CGP:non-semisimple, BCGP}. We prove that this latter class of link invariants behaves well under Murasugi-sum, namely, their top coefficients are Murasugi-sum multiplicative. In particular, this is the case for the Akutsu-Deguchi-Ohtsuki (ADO) invariants \cite{ADO, Murakami:ADO}, a non-semisimple version of the colored Jones polynomials that lie at the core of various recent developments in quantum topology \cite{GPT:modified, CGP:non-semisimple, BCGP, GM:two-variable, Gukov:Coulomb}. This is another result (along with \cite{LNV:genus, LNV:fibred-sl2}) showing that non-semisimple quantum link invariants are more neatly related to geometric topology than their semisimple counterparts, such as colored Jones or HOMFLY. 

\subsection{Main result for ADO} For any $p\geq 2$, the {\em $p$-th ADO invariant}\footnote{Usually \cite{ADO, CGP:non-semisimple, BCGP} the ADO invariant has one variable for every component of the link. The invariant we consider here is obtained by equating all these variables, or equivalently, we use the same color on every component. Because of this, we do not need to multiply by the modified dimensions of \cite{GPT:modified}, so our link invariant is a polynomial instead of a rational function.} is a polynomial invariant of oriented links $L\sb S^3$ with coefficients in the cyclotomic ring $\Z[q^2]$ where $q=e^{\frac{\pi i}{p}}$:
$$\ADO_p(L,t)\in \Z[q^2][t^{\pm \frac{1}{2}}].$$ Let $\top^{ADO}_p(L)$ be the coefficient of $t^{(2g(L)+s-1)(p-1)/2}$ in $ADO_p(L)$, where $s$ is the number of components of $L$ and $g(L)$ the minimal genus of a (connected) Seifert surface of $L$. This is the top coefficient of $ADO_p(L)$ if the genus bound from \cite{LNV:genus} is sharp (and zero if not).

\begin{Theorem}
\label{Theorem: main theorem for sl2 ADO}

If $L$ is a Murasugi sum of $L_1$ and $L_2$ along minimal genus connected Seifert surfaces, then $$\top^{ADO}_p(L)=\top^{ADO}_p(L_1)\cdot\top^{ADO}_p(L_2)$$
for any $p\geq 2$.
\end{Theorem}

It is well-known that the top coefficient of the Alexander polynomial (which coincides with $\ADO_2(L,-t)$) is Murasugi-sum multiplicative (see e.g. \cite{Kindred:simple-proof-Murasugi-thm}). The analogous result for knot Floer homology was recently shown by Cheng-Hedden-Sarkar \cite{CHS:Murasugi-sum-extremal-Floer}. Our theorem (and Theorem \ref{theorem: Murasugi sum multiplicativity of ADO} below) adds an infinite family of Murasugi-sum multiplicative invariants. We are not aware of other link invariants having this multiplicativity property.

\medskip
\def\Ngr{N_{\gg,r}}

\subsection{Higher rank} More generally, we prove a similar theorem for the higher rank analogues of ADO invariants. Let $\gg$ be a simple complex Lie algebra of rank $n$ of type ADE. For any $r\geq 3$ the {\em $r$-th $\gg$-ADO invariant} of an oriented link $L\sb S^3$ is a $n$-variable Laurent polynomial $$\ADOgr(L)\in \Q(q)[t_1^{\pm \frac{1}{2}},\dots, t_n^{\pm \frac{1}{2}}]$$ 
\noindent where $q=e^{\frac{2\pi i}{r}}$. %We will also denote $$\Ngr(L,t):=\ADOgr(L)|_{t_1=\dots=t_n=t}\in \Z[q^2][t^{\pm\frac{1}{2}}].$$ 
This is defined as a Reshetikhin-Turaev invariant using the (ribbon) category of finite-dimensional weight modules over the unrolled restricted quantum group $\UHrqg$ at $q=e^{\frac{2\pi i}{r}}$ \cite{GPT:modified, CGP:non-semisimple, Rupert:unrolled}. This category has a ``continuous'' family of simple modules $\Vla, \lla\in \hh^*$ depending on $n$ complex parameters, which correspond to the variables in $\ADOgr(L)$ (namely $t_i=q^{2\lla(H_i)}$). Let $2\rho=\sum_{\b\in \De^+}\b$ and denote $\tt^{\a}=t_1^{a_1}\dots t_n^{a_n}$ if $\a=\sum_{i=1}^na_i\a_i\in L_R$. 
Let $g(L)$ and $s$ be as before and set $r'=r$ if $r$ is odd, $r'=r/2$ if $r$ is even. We define $$\topgr(L):=\text{ coefficient of }\tt^{(2g(L)+s-1)\cdot (r'-1)\rho} \text{ in } \ADOgr(L).$$  
Equivalently, $\topgr(L)$ is the coefficient of $t^{(2g(L)+s-1)\cdot (r'-1)(\rho,\rho)}$ in the polynomial $\Pgr(L,t)$ obtained by equating all the variables in $\ADOgr$ (Proposition \ref{prop: top of ADOgr is top of Pgr}). These are the top coefficients of $\ADOgr(L)$ or $\Pgr(L)$ if the genus bounds of \cite{LNV:genus, LNV:genus-unrolled} are sharp, and zero if not.  Note that $\ADO_p=\ADO_{\sl_2,2p}$. %These invariants were extended to a TQFT in \cite{dRGP:small-and-unrolled} (and \cite{BCGP} for $\gg=\sl_2$).  

\begin{Theorem}
\label{theorem: Murasugi sum multiplicativity of ADO}
Let $\gg$ be a simple complex Lie algebra of type ADE and $r\geq 3$. If $L$ is a Murasugi sum of $L_1$ and $L_2$ along connected minimal genus Seifert surfaces, then $$\topgr(L)=\topgr(L_1)\cdot \topgr(L_2).$$
\end{Theorem}

As a consequence of the symmetry property for the $\ADOgr$ (Proposition \ref{prop: symmetry of ADO}) the least coefficient (i.e. the coefficient of $\tt^{-(r'-1)(2g(L)+s-1)\cdot \rho}$ in $\ADOgr(L)$) is also Murasugi-sum multiplicative. 
\medskip

The theorem is likely to hold in non-ADE types with appropriate modifications, but we do not do this in this paper.

%Theorem \ref{theorem: Murasugi sum multiplicativity of ADO} also holds for non-simply-laced $\gg$ with some minor (and obvious) restrictions on $r$, namely $r'>\max\{d_i\}_{i=1}^n$ and provided the top coefficient is taken with respect to the variables $t_i=q_i^{\la(H_i)}$ with $q_i=q^{d_i}$. The paper is written in the ADE case, we defer the non-simply-laced case to Subsection \ref{subs: non-simply laced}.

\subsection{Fibred links} A theorem of Giroux-Goodman states that every fibred link can be obtained by plumbing (and possibly deplumbing) Hopf links \cite{GG:fibred}, where plumbing means Murasugi-sum along a square. If $n_{\pm}$ (resp. $m_{\pm}$) is the number of Hopf links $H_{\pm}$ plumbed (resp. deplumbed) then Theorem \ref{theorem: Murasugi sum multiplicativity of ADO} together with the computation of $\topgr(H_{\pm})$ (Subsection \ref{subs: Hopf links computation}) immediately imply:

\begin{Corollary}
\label{corollary: fibred}
    For every simple $\gg$ of type ADE and integer $r\geq 3$, the top coefficient of $\ADOgr(L)$ of a fibred link $L$ is a power of $q$, namely $$\topgr(L)=q^{-(2\rho,2\rho)(n_--m_-)}$$
    where $n_-, m_-$ are as above and $2\rho=\sum_{\b\in\De^+}\b$. 
\end{Corollary}

The number $n_--m_-$ coincides with the Hopf invariant of the plane field of $S^3$ associated to the fiber of the given fibred link, see \cite{NR:difference-index-enhanced-Milnor} or \cite{LNV:fibred-sl2}. The above corollary was proved previously by the authors for $\gg=\sl_2$ at even roots of unity in \cite{LNV:fibred-sl2} (note that top coefficients in \cite{LNV:fibred-sl2} are taken with respect to the variable $x=q^2t$).

\subsection{Integrals of universal invariants} We also prove a formula expressing the top coefficient of $\ADOgr$ in terms of the Hopf algebra integral of the small quantum group $\uqg$ and universal invariants of bottom tangles as in \cite{Habiro:bottom-tangles}. We now suppose $r$ is coprime to the determinant of the Cartan matrix $A=(a_{ij})$ of $\gg$, it is well-known that this guarantees that the small quantum group $\uqg$ is a (finite-dimensional) ribbon Hopf algebra (see e.g. \cite{Turaev:BOOK1}). Thus, there is an associated universal $\uqg$-invariant of bottom tangles as in \cite{Habiro:bottom-tangles}. Given a right integral $\mu\in\uqg^*$, define $\ov{\mu}(-):=\mu(\La_G-)$ where $\La_G$ is the cointegral of the group $G$ of group-likes of $\uqg$ such that $\e(\La_G)=1$. 

\medskip

\begin{Theorem}
\label{theorem: top coefficient formula with integral}
Suppose $(r,\det(A))=1$. Let $\S$ be a connected minimal genus Seifert surface of a link $L$ and let $\TS$ be a (framed) bottom tangle such that $\S$ is obtained by gluing a single disk to the top of $\TS$. Let $Z_{\TS}\in \uqg^{\ot N}$ be the universal $\uqg$-invariant of $\TS$ ($N=2g(L)+s-1$ as before). Then
    $$\topgr(L)=\frac{ \overline{\mu}^{\ot N}(Z_{\TS})}{q^{(2\rho,2\rho)N}\cdot\ov{\mu}(v)^N} 
    %\overline{\mu}^{\ot N}(Z_{\TS})\cdot \frac{q^{-(2\rho,2\rho)N}}{\ov{\mu}(v)^N}
    $$
    where $v$ is the ribbon element of $\uqg$\footnote{In the present paper, the convention is that a positive twist acts by the ribbon element $v\in \uqg$.} (for an explicit expression of $\ov{\mu}(v)$ see Remark \ref{remark: scalar ovmu on ribbon v}).% and $\kappa_{\Lambda}(\tt)\in \Q(q)[q^{\pm\la_1},\dots,q^{\pm \la_n}]$ is defined in Lemma \ref{lemma: CGPT's lemma for cointegral}.
\end{Theorem}

As a consequence of this theorem, the question of sharpness of the genus bounds of \cite{LNV:genus, LNV:genus-unrolled} is related to non-vanishing of universal $\uqg$-invariants of bottom tangles when paired with Hopf integrals. We do not pursue this direction in this paper any further (see Subsection \ref{subs: sharpness of ADO bound} for more comments on this).
\medskip

 For $\gg=\sl_2$ at even roots of unity (so $(r,\det(A))=2$), it is well-known that $\mathfrak{u}_q(\sl_2)$ is not ribbon, but an extension of it is (see \cite{FGST:modular, BBGe:logarithmic}), namely, the small quantum group with group-likes in the weight lattice $\uqLW(\sl_2)$ is ribbon. The above theorem holds provided the universal invariant $Z_{\TS}\in \uqLW(\sl_2)^{\ot N}$ is used. More generally, if $L$ is a lattice with $L_R\sb L\sb L_W$ such that $\uqLg$ is ribbon and the forgetful functor $F:\CC_{[0]}\to \Rep(\uqLg)$ preserves the ribbon structure (where $\CC=\CCgr$), then Theorem \ref{theorem: top coefficient formula with integral} will hold.

\subsection{The $\sl_3$ case at $r=4$} The invariant $\ADO_{\sl_3,4}$, which is a link invariant in $\Z[t_1^{\pm 1},t_2^{\pm 1}]$, was studied by Harper \cite{Harper:sl3-invariant} (note that our $t_i$ is the square of the $t_i$ of \cite{Harper:sl3-invariant}), so our theorems apply to this invariant as well. 

\begin{Corollary}
    The top coefficient of the invariant $\De_{\sl_3}$ of \cite{Harper:sl3-invariant} is Murasugi-sum multiplicative. The top coefficient of $\De_{\sl_3}$ is always $+1$ on fibred links.
\end{Corollary}

The last assertion follows from Corollary \ref{corollary: fibred} together with $(2\rho,2\rho)=8$ for $\sl_3$ and $q^4=1$.

\subsection{Relation to previous works} It is interesting to note that the HOMFLY polynomial $H_L(v,z)$ (where $v$ is the quantum variable) satisfies a similar but more restricted property than our main theorem: the top $z$-term is multiplicative under {\em diagrammatic} Murasugi sum \cite{MP:skein-star-product}. Here top term is also ``diagrammatic'': it is the coefficient of $z^{1-\chi(\S_D)}$ where $D$ is a diagram and $\S_D$ the associated Seifert's algorithm surface (note that by the diagrammatic genus bound of HOMFLY \cite{Morton:Seifert-circles-knot-poly}, this coefficient is zero unless $\S_D$ achieves the canonical genus). This implies, for instance, that the HOMFLY polynomial is $v$-monic on ``diagrammatically fibred'' links (e.g. homogeneous braid closures). However, it is not $v$-monic on all fibred links, the smallest example being $11n74$ which is a genus 2 fibred knot with top HOMFLY term equal to $(1-v^2)z^6$ (see KnotInfo). Hence, the top term of HOMFLY is not Murasugi-sum multiplicative in general.
\medskip

For Murasugi sum along a square (i.e. plumbing) our theorem was already shown for the (top coefficient of the) Links-Gould invariant in \cite{LNV:plumbing-multiplicative}, which is a quantum invariant derived from the quantum supergroup $\Uq(\sl(2|1))$ at generic $q$. The proof given in \cite{LNV:plumbing-multiplicative} relied on certain skein relations for LG, and it is not clear how to extend those techniques to prove the full Murasugi-sum statement (along any polygon) or even how to extend the plumbing result to higher rank. In contrast, the technique of the present paper, which is inspired by that of our previous works \cite{LNV:genus-unrolled, LNV:fibred-sl2}, allows to treat all Lie types in a uniform way. Moreover, an extension of our technique can be adapted to prove the Murasugi-sum statement for Links-Gould, as we do in \cite{PALNV:Murasugi-sum-Links-Gould}. Note that the Links-Gould invariant has the advantage of being defined at {\em generic $q$} which is essential to prove the plumbing-uniqueness result of \cite{LNV:plumbing-multiplicative}. 
\medskip

As mentioned previously, Cheng-Hedden-Sarkar have recently proved a Murasugi-sum multiplicativity theorem for the top knot Floer homology group \cite{CHS:Murasugi-sum-extremal-Floer}. Taking Poincar\'e polynomials, one obtains a $\Z[x^{\pm 1}]$-valued invariant of links which is non-zero on all links (by genus-detection of knot Floer homology) and that is multiplicative under Murasugi-sum. The invariants of the present paper are similar, except that they land in a cyclotomic ring and, more importantly, our invariants can vanish. Thus, we do not get maps on the Murasugi-sum Grothendieck rings defined in \cite{CHS:Murasugi-sum-extremal-Floer}.

%\tcr{Say something about Floer homology result. Note that their result is for index-minimizing surfaces where the index is defined as $i(\S):=\frac{|\p\S|-\chi(\S)}{2}$. Instead we minimize $1-\chi(\S)$ along connected $\S$. The reason for the difference is that knot Floer of links decategorifies to a polynomial of degree $m$ times the multivariable Alexander polynomial, where $m$ is the number of components of the link.}

%It is also interesting to note that a result of Able-Hirasawa \cite{Able-Hirasawa:Murasugi} states that any knot $K$ is a Murasugi sum along non-minimal-genus Seifert surfaces of any two other knots $K_1,K_2$. In contrast, our result obstructs various knots $K$ to be a Murasugi-sum (along minimal genus surfaces) of two given $K_1,K_2$'s (even if the genus of $K$ is the sum of the genus of $K_1,K_2$ and the Alexander obstruction gives no information).\medskip

\subsection{Sketch of proof ($\sl_2$ case)} We give a sketch of the proof of Theorem \ref{Theorem: main theorem for sl2 ADO}. If $\S$ is a Murasugi-sum of $\S_1$ and $\S_2$ along a polygon $P$ and $L=\p\S, L_i=\p\S_i$, then the Reshetikhin-Turaev operator invariant of $L_o$ (the link $L$ cut open at the polygon) colored with $V_{\lambda}$ can be expressed as $$F(L_o,V_{\lambda})=(\id_{V_{\lambda}}\ot \FF(T_{\S_2},X'))\circ I \circ (\FF(T_{\S_1},X)\ot\id_{V_{\lambda}})$$
where $T_{\S_1}$ is a bottom tangle associated to $\S_1$ colored with $X=V_{\lambda}\ot V_{\lambda}^*$, while $T_{\S_2}$ is a ``top tangle'' associated to $\S_2$ colored with $X'=V_{\lambda}^*\ot V_{\lambda}$ (see Figure \ref{figure: Murasugi sum in terms of bottom and top tangles}). Then the idea is to apply Propositions \ref{prop: prop surface and coevaluations} and \ref{prop: F(TS,X) as coev but upside-down} which essentially say that 
\begin{align}
    \label{eq: intro F(TS,X) prop}
    \FF(T_{\S_1},X)=\top_p^{ADO}(L_1)\cdot t^{\frac{p-1}{2}N_1}\cdot\lcoev^{\ot 2N_1}+ \text{ lower degree terms }
\end{align}
where $N_1=\dim H_1(\S_1)$ and similarly for $\FF(T_{\S_2},X')$, which (after some care with the meaning of ``lower degree terms'') imply that $$F(L_o,V_{\lambda})=\top_p^{ADO}(L_1)\cdot\top_p^{ADO}(L_2)\cdot t^{\frac{p-1}{2}\dim H_1(\S)}\cdot\id_{V_{\lambda}}+ \text{ lower degree terms}$$
implying Theorem \ref{Theorem: main theorem for sl2 ADO}.
\medskip

Thus the core of the argument is Proposition \ref{prop: prop surface and coevaluations} which is proved by pushing further the technique of \cite{LNV:genus-unrolled}, which is based on universal invariants of bottom tangles from small quantum groups, namely:

\def\uq{\mathfrak{u}_q(\sl_2)}
\def\uqLW{\mathfrak{u}^{L_W}_q(\sl_2)}

\begin{enumerate}
    \item Write the invariant $\ADO_p(L)$ in terms of the action of a universal $\uqLW$-invariant $Z_{T}\in \uqLW^{\ot N}$ acting on the module $X^{\ot N}$.
    \item Express $Z_{T}$ in the PBW basis. This  can be written as $Z_{T}=Z^{top}+Z^{low}$ where $Z^{top}$ has the maximal powers of $E$ (which is $E^{p-1}$ since $E^p=0$).
    \item By the algebraic bounds of \cite{LNV:genus-unrolled} (Lemma \ref{lemma: algebraic bounds}), $Z^{low}$ does not contributes to the top coefficient of ADO, hence we only have to look at $Z^{top}$.
    \item $Z_{T}$ has degree zero (where $|E|=1=-|F|$) since the $R$-matrix does, hence $Z^{top}$ has the form $$Z^{top}=\otimes_{i=1}^Nf_i(K)F^{p-1}E^{p-1}.$$
    \item By Lemma \ref{lemma: change FEmax by cointegral}, we can replace $F^{p-1}E^{p-1}$ by the cointegral $\La$ of $\uqLW$ up to lower degree terms and by Lemma \ref{lemma: CGPT's lemma for cointegral} (from \cite{CGPT}), the cointegral has the effect of cutting all the bands, which brings the tensor into the form in (\ref{eq: intro F(TS,X) prop}).
    \item Theorem \ref{theorem: top coefficient formula with integral} follows by tracing down the top coefficient along this argument.
\end{enumerate}

This argument relies almost solely on the Hopf algebra structure of $\uqLW$, that's why it can be generalized to higher rank almost word by word. We do not even need $\uqg$ to be ribbon, as long as a lattice $L\supseteq L_R$ has been chosen so that the braiding $c_{X,X}$ can be expressed as left multiplication by an element of $\uqL(\gg)^{\ot 2}$ (Lemma \ref{lemma: braiding on cXX expressed as R-matrix on some small uqgL}) which is enough to run the universal-invariant argument. This is why Theorem \ref{theorem: Murasugi sum multiplicativity of ADO} holds with no condition between $r$ and $\det(A)$. Note that in the above argument we didn't really used that surfaces have minimal genus. However, by the genus bounds, the corresponding coefficients would all be zero, so no non-trivial statement would be obtained.

\subsection{Plan of the paper} In Section \ref{sect: Murasugi sum} we define Murasugi sum. In Section \ref{sect: unrolled quantum groups} we recall various known facts about unrolled quantum groups, the ribbon structure of the category $\CCgr$ and the modules $\Va$. In Section \ref{sect: link invariants} we define the link polynomials $\ADOgr$, we prove some basic properties such as symmetry and recall the genus bound of \cite{LNV:genus-unrolled}. In Section \ref{sect: algebraic lemmas} we prove the purely algebraic lemmas necessary for Theorem \ref{theorem: Murasugi sum multiplicativity of ADO}. These are lemmas about the action of the quantum group on the module $X=\Va\ot\Va^*$, notably the action of the cointegral of $\uqL(\gg)$ (or $\uqg$ if ribbon). The proofs of Theorems \ref{theorem: Murasugi sum multiplicativity of ADO} and \ref{theorem: top coefficient formula with integral} are carried out in Section \ref{sect: proofs main thms}.

%\tcr{Use height for sum of coefs in roots.}

\subsection{Acknowledgements} The first author thanks Christian Blanchet for various interesting conversations and Bertrand Patureau-Mirand for email communication. The first author acknowledges support from Fondecyt INI 11261601. AI was used to make minor edits at the very end of the work. It was also used to try to remove the connected hypothesis in the main theorems, but this was unsuccessful. No writing was done by AI. Pictures were made with LatexDraw.

\section{Murasugi sum}
\label{sect: Murasugi sum}
\def\inte{\text{int}}

Let $\S_1,\S_2$ be two compact, oriented, embedded surfaces-with-boundary in $S^3$ that are separated by a plane. For each $i=1,2$, choose an embedded $2n$-polygon $P_i\sb \S_i$ ($n\geq 1$ is the same for both polygons) such that the $2n$-sides of $P_i$ alternate between the boundary and the interior of $\S_i$. A {\em $2n$-Murasugi sum} of $\S_1$ and $\S_2$ is any surface $\S$ obtained by gluing $P_1$ to $P_2$ with an orientation-preserving homeomorphism in such a way that the arcs in $P_1\cap \p\S_1$ (resp. $P_1\cap \inte(\S_1)$) are glued to the arcs in $P_2\cap \inte(\S_2)$ (resp. $P_2\cap \p\S_2$). Here $\inte(\S_i)$ denotes the interior of the surface $\S_i$. This last condition guarantees that the resulting surface $\S$ is a smooth oriented surface. See Figure \ref{fig:Murasugi sum}. 

\begin{figure}[h]
    \centering
    \includegraphics[width=9cm]{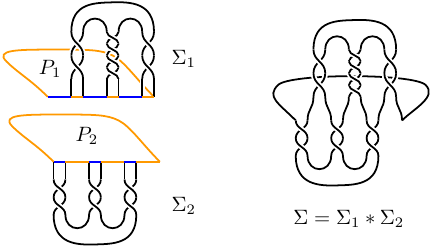}
    \caption{A Murasugi sum $\S$ of two pretzel surfaces $\S_1=P(-2,-4,-2)$ and $\S_2=P(-2,-2,-2)$ along hexagons $P_i\sb \S_i$. The long orange arc of $P_1\cap \p\S_1$ has to be though as an arc in $\inte(\S_1)$ very close to  $\p\S_1$. In this way, the hexagon $P_1$ is glued vertically to $P_2$ in such a way that orange arcs of $P_1\cap \inte(\S_1)$ are glued to orange arcs of $P_2\cap \p\S_2$ and similarly for blue arcs.}
    \label{fig:Murasugi sum}
\end{figure}

We denote $\S=\S_1*\S_2$ any surface obtained by Murasugi-summing $\S_1$ and $\S_2$. Of course, this notation is ambiguous, as usually there are infinitely many different ways to Murasugi-sum two surfaces (except in the case that both $\S_1$ and $\S_2$ are annuli).

\medskip
As mentioned in the introduction, the main theorem about Murasugi sum is the following theorem of Gabai \cite{Gabai:foliations-genera, Gabai:Murasugi1}:

\begin{theorem}(Gabai)
    Let $\S$ be a Murasugi sum of $\S_1$ and $\S_2$. Then $\S$ is a genus-minimizing Seifert surface for the link $L=\p\S$ if and only if each $\S_i$ is a genus-minimizing Seifert surface for $L_i=\p\S_i$, $i=1,2$.
\end{theorem}

Thus, we say that a link $L$ is a {\em Murasugi sum of links} $L_1,L_2$ if $L$ bounds a Seifert surface $\S$ that is a Muraugi sum of genus-minimizing Seifert surfaces $\S_1,\S_2$ for $L_1,L_2$ respectively. This depends on even more choices as before since minimal genus Seifert surfaces of links are not necessarily unique (this is measured by a complex known as the Kakimizu complex).

\section{Unrolled quantum groups}
\label{sect: unrolled quantum groups}

In this section we define quantum groups and their unrolled versions, we discuss PBW basis, the modules $\Vla$ and the ribbon category structure on $\CCgr$. For more details on unrolled quantum groups see \cite{GP:nonrestricted, GP:trace-projective-quantum-groups} which deals with the odd root of unity case though in the more general ``unrestricted'' case, or \cite{Rupert:unrolled} which deals with even roots of unity case in the restricted case. In the present paper we consider the restricted case in both even and odd roots of unity. The $\sl_2$ case (restricted unrolled at even roots of unity) is well-developed in various works such as \cite{CGP:non-semisimple, BCGP}.
\medskip

\begin{notation}
In what follows, $\gg$ denotes a finite-dimensional complex simple Lie algebra of rank $n$, $\hh$ a Cartan subalgebra of $\gg$, $\{\a_1,\dots,\a_n\}$ a set of positive simple roots, $(a_{ij})$ its Cartan matrix and $\Dep$ the set of positive roots. We work in the type ADE case only, so $(a_{ij})$ is symmetric. We define $H_1,\dots,H_n\in \hh$ by $\a_i(H_j)=a_{ij}$ for all $i,j$. We let $(\a_i,\a_j)=a_{ij}$ and extend it to a form $(-,-)$ on $\hh^*$ so that $(\la,\a_i)=\la(H_i)$ for all $\lla\in\hh^*$. We denote by $L_R, L_W\sb \hh^*$ the root and weight lattices of $\gg$ respectively. We denote $\rho=\frac{1}{2}\sum_{\b\in\Dep}\b\in L_W$. We also denote by $L_R^+$ the set of positive vectors in $L_R$ i.e. those vectors $\sum_{i=1}^nk_i\a_i$ with all $k_i\geq 0$ and some $k_j>0$ and for any $\c_1,\c_2\in L_R$ we write $\c_1> \c_2$ if $\c_1-\c_2\in L_R^+.$ \\

Let $q$ be a primitive root of unity of order $r$, say $q=e^{\frac{2\pi i}{r}}$, and for every $z\in \C$ set $q^z=e^{\frac{2\pi i z}{r}}$ and $[z]_{\qi}=\frac{\qi^z-\qi^{-z}}{\qi-\qi^{-1}}$. For $k\in\N$ denote $[k]_q!=[k]_q[k-1]_q\cdots [1]_q$. Let $r'=r$ if $r$ is odd and $r'=r/2$ if $r$ is even.  \\

For any $\lla\in\hh^*$ we let $\la_i=\la(H_i)$. The reader can assume that $\la_i\in \C\sm \Z$ for every $i$, though $\la$ typical is enough. As mentioned above, we set $t_i=q^{2\la_i}$ and $\tt^{\a}=t_1^{a_1}\dots t_n^{a_n}$ for any $\a=\sum a_i\a_i\in L_R$. With this notation we have $q^{(r'-1)(\la,2\rho)}=\tt^{(r'-1)\rho}$.\\

Some proofs use graphical notation. We use the convention that diagrams are read from bottom to top, so an arrow colored by $V$ represents $\id_V$ (resp. $\id_{V^*}$) if oriented upwards (resp. downwards).

\end{notation}

\subsection{Unrolled quantum groups}
\label{subs: unrolled quantum groups}
Let $n=\rank(\gg)$ and $q\in \C$ with $q\neq \pm 1$. The {\em unrolled quantum group} associated to $\gg$, denoted $\UHqg$, is the $\C$-algebra with generators $E_i,F_i,K_i^{\pm 1}, H_i$, $i=1,\dots,n$ and relations
\begin{align*}
[H_i,E_j]&=a_{ij}E_j, & [H_i,F_j]&=-a_{ij}F_j, & [H_i,K_j^{\pm 1}]&=0=[H_i,H_j],\\
K_iE_j&=\qi^{a_{ij}}E_jK_i, & K_iF_j&=\qi^{-a_{ij}}F_jK_i & K_iK_j&=K_jK_i, \\
[E_i,F_j]&=\d_{ij}\frac{K_i-K_i^{-1}}{\qi-\qi^{-1}}, & K_iK_i^{-1}&=K_i^{-1}K_i=1, & & 
\end{align*}
for every $i,j$ and
\begin{align*}
  \sum_{k=0}^{1-a_{ij}}(-1)^k\left[\begin{matrix} 1-a_{ij} \\ 
 k\end{matrix}\right]_{\qi}E_i^{1-a_{ij}-k}E_jE_i^k&=0,  \\
  \sum_{k=0}^{1-a_{ij}}(-1)^k\left[\begin{matrix} 1-a_{ij} \\ 
 k\end{matrix}\right]_{\qi}F_i^{1-a_{ij}-k}F_jF_i^k&=0, 
\end{align*}
for $i\neq j$. This is a Hopf algebra if we define 
\begin{align*}
\De(E_i)&=E_i\ot K_i+1\ot E_i, & \De(F_i)&=K_i^{-1}\ot F_i+F_i\ot 1, \\
 \De(K_i)&=K_i\ot K_i, & \De(H_i)&=H_i\ot 1+1\ot H_i
\end{align*}
for every $i$. For $\a=a_1\a_1+\dots+a_n\a_n\in L_R$ we denote $K_{\a}:=K_1^{a_1}\dots K_n^{a_n}$.

\medskip

We denote by $\Uqg\sb \UHqg$ the subalgebra generated by $E_i,F_i,K_i^{\pm 1}$. We denote by $\Uh$ the subalgebra of $\UHqg$ generated by the $H_i$'s. Let $U^+$ (resp. $U^-$) be the subalgebra generated by $E_1,\dots,E_n$ (resp. $F_1,\dots,F_n$). We also denote by $U^{\geq 0}$ (resp. $U^{\leq 0}$) the subalgebra generated by the $K_i^{\pm 1}$'s and $U^+$ (resp. $K_i^{\pm 1}$'s and $U^-$).

\medskip

The algebra $\UHqg$ is $L_R$-graded if we set $d(E_i)=\a_i, d(F_i)=-\a_i, d(H_i)=d(K_i)=0$ for each $i$ and we extend multiplicatively. This way $\Uqg$ becomes a $L_R$-graded Hopf algebra.

\subsection{PBW basis} 

Let $\Dep=\{\b_1,\dots,\b_N\}$ be the set of positive roots of $\gg$ which we assume is ordered according to a reduced expression of the longest element of the Weyl group. The root vectors $E_1,\dots, E_n$ correspond to the simple roots $\a_1,\dots,\a_n$. As usual, one can define root vectors $E_{\b}$ and $F_{\b}$ for all $\b\in \Dep$, see \cite{CP:BOOK} for more details. Then for every function $I:\Dep\to \N_{\geq 0}$ we define $$E^I:=E_{\b_N}^{I(\b_N)}\cdots E_{\b_1}^{I(\b_1)}, \ \ F^I:=F_{\b_N}^{I(\b_N)}\cdots F_{\b_1}^{I(\b_1)}.$$
If for every $I$ we denote $$d(I)=\sum_{\b\in\Dep}I(\b)\b\in L_R,$$ then $d(E^I)=d(I)$ and $d(F^I)=-d(I)$. We also denote $$|I|=\sum_{\b\in\De^+}I(\b)|\b|$$
where we set $|\a_i|:=1$ for positive simple roots and define $|.|$ additively on all positive roots.

\begin{proposition}\cite{CP:BOOK}
\label{lemma: PBW belongs to V+}
For each positive root $\b$, $E_{\b}\in U^+$ and $\{E^I \ | \ I:\Dep\to \N_{\geq 0}\}$ is a basis of $U^+$. Similarly $F_{\b}\in U^-$ and $\{F^I \ | \ I:\Dep\to \N_{\geq 0}\}$ is a basis of $U^-.$ 
\end{proposition}

\subsection{Unrolled restricted and small quantum groups} 

%For each $i=1.\dots,n$ let $r_i$ be the order of $\qi^2$. For $\b\in\Dep$ let $\rb=r_i$ if $\b$ is in the same Weyl group orbit as $\a_i$. The {\em restricted} unrolled quantum group is the quotient $\UHr$ of $\UHqg$ by the two-sided ideal generated by $E_{\b}^{\rb}$ and $F_{\b}^{\rb}$ for each $\b\in\Dep$. Note that this ideal is a Hopf ideal, this is because $E_i\ot K_i$ and $1\ot E_i$ $\qi^2$-commute. 

The {\em restricted} unrolled quantum group is the quotient $\UHr$ of $\UHqg$ by the two-sided ideal generated by $E_{\b}^{r'}$ and $F_{\b}^{r'}$ for each $\b\in\Dep$. Here $r'=r$ if $r$ is odd and $r'=r/2$ if $r$ is even. This ideal is a Hopf ideal, hence the quotient is a Hopf algebra. The subalgebra of $\UHr$ generated by $E_i,F_i,K_i^{\pm 1}$ is denoted $\Ur$ and simply called the restricted quantum group. The {\em small quantum group }$\uqg$ is the quotient of $\Ur$ by setting $K_i^r=1$ for all $i=1,\dots, n$. Let $Q_r$ be the subset of $L_R$ of linear combinations $c_1\a_1+\dots+c_n\a_n$ with $0\leq c_i\leq r-1$. %The {\em small Borel} is the subalgebra $\uqb$ of $\uqg$ generated by $K_i^{\pm 1}, E_i$ for $i=1,\dots,n$.
\medskip

\begin{proposition}
    The set $$\{K_{\a}F^IE^J \ | \ \a\in Q_r \text{ and } I,J:\Dep\to\{0,\dots,r'-1\}\}$$ forms a basis of $\uqg$.
\end{proposition}

We will denote $\Imax$ the function defined by $\Imax(\b)=r'-1$ for all $\b\in \Dep$. Thus, $E^{\Imax}$ is the element of maximal $\N$-degree in the Borel subalgebra of $\uqg$. Note that $$d(\Imax)=(r'-1)2\rho, \ |\Imax|=(r'-1)(2\rho,\rho).$$

\subsection{Weight modules}

A {\em weight module} is a $\UHr$-module for which the action of each $H_i$ is diagonalizable and $K_i$ acts as $\qi^{H_i}$ for every $i$. We denote by $\CCgr$ the category of finite-dimensional weight modules over $\UHr$. 

\medskip
Let $V$ be a weight module. For each $\lla$ in the dual Cartan $\hh^*$ let $$V(\lla)=\{v\in V \ | \ Hv=\lla(H)v, \forall H\in \hh\}.$$
These are the vectors of weight $\lla$ in $V$. Clearly $E_i(V(\lla))\sb V(\lla+\a_i)$ and $F_i(V(\lla))\sb V(\lla-\a_i)$ so for any $\lla\in\hh^*$, $\oplus_{\c\in L_R}V(\lla+\c)$ is a submodule of $V$. Thus, every simple weight module has all its weights in $\lla+L_R$ for some $\lla\in\hh^*$.
\medskip
\def\mmu{\boldsymbol{\mu}}
  
 Let $G=\hh^*/L_R$ and let $\CC_{[\lla]}$ be the category of weight modules whose weights are in the class of $[\lla]:=\lla+L_R\in G$. Then the whole category $\CCgr$ of weight modules splits as a direct sum of subcategories $$\CCgr=\bigoplus_{[\lla]\in G}\CC_{[\lla]}.$$
Moreover, one has $\CC_{[\lla]}\ot\CC_{[\mmu]}\sb \CC_{[\lla]+[\mmu]}$ for each $[\lla],[\mmu]\in G$, i.e., $\CCgr$ is a $G$-graded tensor category.

\begin{remark}
    The group $G$ is isomorphic to a torus $(\C^{\t})^n$ where $n=\rank(\gg)$. Thus, we like to call $\CCgr$ a {\em toric tensor category}. Though we do not explicitly use it, this torus grading is essential for our results and is the main difference with quantum groups (of Lie algebras) at generic $q$. It is very interesting to note that, for some quantized Lie superalgebras, the toric structure is present even when $q$ is generic. This is the case for the ``type I'' Lie superalgebras (except $\mathfrak{psl}(n|n)$), and its the reason why analogous results can be expected for them (see \cite{GLN:genus-for-Lie-superalgebras} for genus bounds). 
\end{remark}

 %This functor is essentially surjective if $gcd(r,\det(A))=1$, meaning, every object of $\Rep(\uqg)$ is isomorphic to an object in the image of $F$.

\subsection{The typical modules $\Vla$ and $V_{\la}$} For each $\lla\in \hh^*$, let $\Vla$ be the (finite dimensional) $\UHrqg$-module generated by a highest weight vector $v_0$ of highest weight $\lla$ (highest with respect to the partial order of $L_R$). In other words $$\Vla=\ov{U}_q(\gg)\ot_{U^{\geq 0}}\C$$ where $\C$ is a $U^{\geq 0}$-module via the map $\xi:U^{\geq 0}\to\C$ defined by $\xi(E_i)=0$ and $\xi(K_i)=q^{\lla(H_i)}$ for every $i$. Whenever $\lla$ is {\em diagonal}, meaning $$\lla(H_1)=...=\lla(H_n)=\la$$ for some $\la\in \C$ we will simply denote $V_{\lla}$ by $V_{\la}$. Note that $\tt^{\alpha}=q^{\la(\a,\rho)}$ for diagonal $\lla$ (at least in type ADE, since we use $(\rho,\a_i)=2$ for all $i$).  %Here $B$ denotes the (Borel) subalgebra of $\ov{U}_q(\gg)$ generated by the $K_i^{\pm 1}, E_i$ for $i=1,\dots,n$. 
 \medskip

We call $\lla$ {\em typical} if $q^{2( \lla+\rho,\b)-m(\b,\b)}\neq 1$ for all $\b\in\Dep$ and $m=1,\dots,r'-1$. This is equivalent to say that the vectors $v_I:=F^Iv_0$, for $I:\Dep\to \{0,\dots, r'-1\}$, form a basis of $\Vla$, see \cite{GP:nonrestricted}. We will call this the {\em standard basis} of $\Va$. The standard basis of $\Va^*$ is simply the dual of this basis.
\medskip

Note that $v_I$ has weight $\lla-d(I)$. Hence, the highest weight of $\Vla$ is $\lla$ and the lowest weight is $\lla-d(I_{max})=\lla-(r'-1)2\rho$ (with respect to the $L_R$-order). 

\begin{proposition}\cite{GP:trace-projective-quantum-groups, Rupert:unrolled}
    For typical $\lla$, $\Vla$ is simple and projective in $\CCgr$. In particular, if $\lla$ is diagonal, then $V_{\la}$ is simple and projective.
\end{proposition}

We will be concerned with the module $X=\Vla\ot \Vla^*$. We will always consider this with the basis that is the tensor product of the standard basis of $\Va$ and $\Va^*$. Note that the weights of $X$ have the form $$\lla-d(I)+(-\lla+d(J))=-d(I)+d(J) \in L_R\sb \hh^*$$
so that the highest and lowest weights of $X$ are $\pm d(I_{max})$.

\subsection{Pivotal structure} The category $\CCgr$ is pivotal with left evaluations and coevaluations the usual ones of vector spaces, and right evaluations and coevaluations given by $$\rev_V(v\ot v')=v'(gv), \ \ \rcoev_V(1)=\sum v_i\ot g^{-1}v_i'$$
where $g=K_{2\rho}^{1-r'}$ (recall $2\rho=\sum_{\b\in \De^+}\b$). Here $(v_i)$ is a basis of $V$ and $(v_i')$ is the dual basis. It follows that there is an isomorphism $j_V:V^{**}\to V$ such that $$j_V^{-1}(v)(v')=v'(gv).$$
For $\Vla$ one has 
\begin{align}
    \label{eq: pivotal on Verma}
    j_{\Vla}(v''_I)=q^{(r'-1)(\la-d(I),2\rho)}v_I=\tt^{(r'-1)\rho}q^{-(r'-1)(d(I),2\rho)}v_I
\end{align}
\noindent where $v''_I$ is the basis of $\Vla^{**}$ determined by $v''_I(v'_J)=\d_{I,J}$.

\medskip

Let $X=\Vla\ot \Vla^*$. Then $X^*\cong \Vla^{**}\ot \Vla^*\cong X$ via the pivotal structure of $\CCgr$, denote by $\iota_X:X^*\to X$ such isomorphism. By (\ref{eq: pivotal on Verma}), this is given by 
\begin{align}
\label{eq: j on X}
    \iota_X((v_I\ot v_J')')=j_{\Vla}(v''_J)\ot v'_I=q^{(r'-1)(\lla-d(J),2\rho)}v_J\ot v'_I=\tt^{(r'-1)\rho}q^{n_J}v_J\ot v'_I
\end{align}
for some $n_J\in\Z$.  In particular, we have:

\begin{lemma}
\label{lemma: coeffs of iota X}
    If $\lla\equiv \la$ is diagonal, then $\iota_X$ has coefficients in $t^{(r'-1)(\rho,\rho)}\Q(q)$ in the standard bases of $X^*$ and $X$ (here $t=q^{2\la}$).
\end{lemma}
\medskip

\begin{lemma}
\label{lemma: coev lemma}
    We have $(\iota_X\ot (\lcoev_{\Va}\circ\rev_{\Va}))\circ\rcoev_X=\lcoev_{\Va}\ot\lcoev_{\Va}$
    or graphically
    \begin{figure}[H]
        \centering
        \includegraphics[width=6cm]{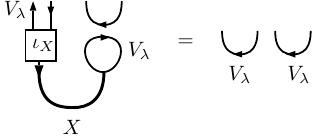}
    \end{figure}
    where the thick strand is colored by $X$ and all other strands are colored with $\Vla$.
\end{lemma}
\begin{proof}
    Indeed
    \begin{align*}
    (\iota_X\ot (\lcoev_{\Va}&\circ\rev_{\Va}))\circ\rcoev_X(1)\\
    &=\sum_{I,J} q^{(r'-1)(2\rho, d(J)-d(I))}\iota_X(v_I\ot v_J')'\ot (\lcoev_{\Va}\circ\rev_{\Va})(v_I\ot v'_J)\\
     &=\sum_{I,J,C} q^{(r'-1)[(2\rho, d(J)-d(I))+(2\rho,\lla-d(J))]}v_J\ot v'_I\ot \d_{I,J}q^{(1-r')(2\rho,\lla-d(I))}v_C\ot v'_C\\
     &=\sum_{I,C}v_I\ot v'_I\ot v_C\ot v'_C\\
     &=\lcoev_{\Va}\ot \lcoev_{\Va}.
    \end{align*}
\end{proof}

\subsection{Braiding on weight modules}
\label{subs: braiding on weight} 
The category of weight $\UHr$-modules is braided with $c_{V,W}:V\ot W\to W\ot V$ defined by $$c_{V,W}(v\ot w)= \tau_{V,W}(\HH\TT(v\ot w))$$
where $$\TT=\sum_{I:\Dep\to\{0,\dots,r'-1\}} c_{I} \cdot E^I\ot F^I$$
with $$c_I=\prod_{\a\in\Dep}\frac{(\qa-\qa^{-1})^{I(\a)}}{[I(\a)]_q!}\qa^{\frac{I(\a)(I(\a)-1)}{2}}\in \Q(q),$$ and where
$\HH$ denotes the operator $V\ot W\to V\ot W$ defined by $$\HH(v\ot w)=q^{(\lla,\mmu)}v\ot w$$ if $v\in V(\lla)$ and $w\in W(\mmu)$. Note that for $I=\Imax$ one has
\begin{equation}
\label{eq: formula for cImax}
    c_{\Imax}=\left(\frac{(q-q^{-1})^{r'-1}}{[r'-1]_q!}q^{\frac{(r'-1)(r'-2)}{2}}\right)^{|\De_+|}=\left(\frac{q^{1-r'}}{r'}(q-q^{-1})^{2(r'-1)}\right)^{|\De_+|}
\end{equation}
where we used that $\prod_{k=1}^{r'-1}(q^k-q^{-k})=q^{\frac{r'(r'-1)}{2}}r'$ which easily follows by writing $X^{r'}-1$ as a product of linear factors according to roots (which are the $q^{-2i}$ with $i=0,\dots,r'-1$).
\medskip

\subsection{Ribbon structure}\label{subs: ribbon structure} The category $\CCgr$ is ribbon (see \cite{GP:nonrestricted, GP:trace-projective-quantum-groups} for odd roots, \cite{Rupert:unrolled} for even roots and \cite{CGP:non-semisimple} or \cite{Ohtsuki:BOOK} for the $\sl_2$ case at even roots) in a way compatible with the above braiding and pivotal structure. The twist on a typical (simple) $\Vla$ is given by 
\begin{align}
\label{eq: ribbon twist on Verma}
    \theta_{\Vla}=q^{(\lla,\lla-(r'-1)2\rho)}\cdot\id_{\Vla}=q^{(\lla,\lla)}\tt^{-(r'-1)\rho}\cdot\id_{\Vla}.
\end{align}

\subsection{$\CC_{[0]}$ and small quantum groups} As noted above, the category $\CCgr$ is graded by $G=\hh^*/L_R$ as a tensor category. Since $\CCgr$ is ribbon, it follows that the neutral component $\CC_{[0]}$ is ribbon as well. This category consists of weight modules with weights in $L_R$. Note that this is not a finite tensor category, for instance, the trivial $U_q(\sl_2)$-module becomes a module in $\CC_{[0]}$ in infinitely many distinct ways by setting $Hv_0=2rkv_0$ with $k\in \Z$. 
\medskip

\def\uqwg{\mathfrak{u}_q^{L}(\gg)}
\def\uqL{\mathfrak{u}_q^{L}(\gg)}

Let $\uqwg$ be the extension of the small quantum group whose Cartan part is $\C[L]$ where $L$ is a lattice such that $L_R\sb L\sb L_W$. In other words, $\uqwg$ has generators $E_i,F_i$ for $i=1,...,n$ as before and Cartan generators $K_{\c}$ for each $\c\in L$ satisfying $K_{\c}K_{\c'}=K_{\c+\c'}$ and $$K_{\c}E_i=q^{(\c,\a_i)}E_iK_{\c}, \ K_{\c}F_i=q^{-(\c,\a_i)}F_iK_{\c}, $$
as well as $K_{\b}^r=1$ for each $\b\in L_R$, plus the other usual quantum group relations (among them $E_{\a}^{r'}=0=F_{\a}^{r'}$ for all $\a\in\De^+$). Note that since $[L_W:L_R]$ is finite, $\uqL$ is finite-dimensional. Any module in $\CC_{[0]}$ becomes a module over $\uqL$ in a natural way, in fact one can take $L=L_W$: simply set $K_{\a}v=q^{(\a,\lla)}v$ for each $\a\in L_W$ and $v\in V(\lla)$. Since $(\a,\lla)\in \Z$ for $\a\in L_W,\lla\in L_R$ it follows that $K_{\a}^{2r}$ acts as the identity. Thus there is a forgetful functor $F:\CC_{[0]}\to \Rep(\uqL)$.
\medskip

For each typical $\la\in \hh^*$, the module $X=\Va\ot\Va^*$ is in $\CC_{[0]}$. Thus, we can consider $X$ as a module over $\uqL$ and the braiding $c_{X,X}$ as a $\uqL$-morphism.

\def\wH{\widehat{H}}
\def\uu{\boldsymbol{u}}
\def\vv{\boldsymbol{v}}

\begin{lemma}
\label{lemma: braiding on cXX expressed as R-matrix on some small uqgL}
For each $\gg$ of type ADE and $r\geq 3$ there is a lattice $L$ (that depends on $r$) with $L_R\sb L\sb L_W$ having the following property:
there exists $\ov{\KK}_L\in \Q(q)[L]^{\ot 2}$ such that for every $V,W\in \CC_{[0]}$ the braiding $c_{V,W}$ coincides with left multiplication by $\ov{\KK}_L\cdot \TT$, where $\TT\in \uqL^{\ot 2}$ is as before: $$c_{V,W}(v\ot w)=\tau_{V,W}(\ov{\KK}_L\cdot\TT(v\ot w)).$$ 
%the braiding $c_{X,X}:X\ot X\to X\ot X$ of the category $\CC_{[0]}$ can be expressed as  $$c_{X,X}(v\ot w)=\tau_{V,W}(\ov{\KK}\cdot\TT(v\ot w))$$where $\TT\in \uqL^{\ot 2}$ is as before and $\ov{\KK}\in \C[L]^{\ot 2}$. 
Moreover, when $(r,\det(A))=1$ one can take $L=L_R$ and $\ov{\KK}\cdot \TT$ to be the $R$-matrix of $\uqg$. When $\gg=\sl_2$ and $r$ is even one can take $L=L_W$ and $\ov{\KK}\cdot\TT$ the $R$-matrix of $\uqLW(\sl_2)$ of \cite{FGST:modular, BBGe:logarithmic}.
\end{lemma}

Note that this lemma does not implies that $\uqL$ is ribbon. We do not claim that $\ov{\KK}\cdot \TT$ is an $R$-matrix for $\uqL$, but that it works as an $R$-matrix for all objects in the image of the forgetful functor $F:\CC_{[0]}\to \Rep(\uqL)$.

\begin{proof}
By the formula for the braiding in $\CC_{[0]}\sb \CCgr$, the lemma amounts to prove that the diagonal operator $\HH$ can be written as left multiplication by $\ov{\KK}\in \Q(q)[L]^{\ot 2}$. We prove it for $L=L_W$ (which works for all $\gg$ and $r$). Then, we want to show that there are coefficients $c_{\a,\b}\in\Q(q)$ such that $$q^{(\la,\mu)}=\sum_{\a,\b\in L_W}c_{\a,\b}\cdot q^{(\a,\la)+(\b,\mu)}$$
for all $\la,\mu\in L_R$. Writing $\la=\sum a_i\a_i, \mu=\sum b_i\a_i$ and $\a=\sum u_i\wH_i, \b=\sum v_i\wH_i$ ($a_i,b_i,u_i,v_i\in \Z$) where $(\wH_i)\sb L_W$ is the dual basis of $(H_i)$, the above equation becomes 
\begin{align}
    \label{eq: qbAc as a sum}
    q^{\sum a_{ij}a_ib_j}=\sum_{\uu,\vv\in L_W} c_{\uu,\vv}\cdot q^{\sum u_ia_i+v_ib_i}.
\end{align}
For this, it suffices to prove that, given a fixed $a\in \Z$ there are coefficients $c_{i,j}\in \Q(q)$ such that 
\begin{align}
    \label{eq: writing qabc as a sum}
    q^{a(bc)}=\sum_{i,j}c_{ij}\cdot q^{bi+cj}
\end{align}
for all $b,c\in\Z$. For $a=0$ this is obvious (set $c_{00}=1$ and $c_{ij}=0$ otherwise).
When $(r,a)=1$ it suffices to set $c_{ij}=q^{-a'ij}$ where $a'$ is the inverse mod $r$ of $a$, and sum over $i,j\in \{0,\dots,r-1\}$. If $(r,a)=d>1$ then $\o=q^a$ is a root of 1 of order $r/d$ so what we just showed proves that $\o^{bc}=\sum c'_{i',j'}\o^{bi'+cj'}$ for some $c'_{i',j'}\in\Q(q)$. Setting $c_{i,j}=0$ if either $i,j$ are not divisible by $a$ and $c_{ai',aj'}=c'_{i',j'}$ proves (\ref{eq: writing qabc as a sum}). Applying (\ref{eq: writing qabc as a sum}) $n^2$ times to $q^{\sum a_{ij}a_ib_j}=\prod_{i,j}q^{a_{ij}a_ib_j}$ gives an expression of the form (\ref{eq: qbAc as a sum}).

\medskip
    The statement for $(r,\det(A))=1$ is proven in \cite[Lemma 2.4]{LNV:genus-unrolled}. The statement for $\sl_2$ is well-known and proved similarly.
\end{proof}

\medskip

\def\Im{I_{max}}

\section{Link invariants}
\label{sect: link invariants}
\def\Proj{\text{Proj}}

We assume the reader is familiar with the Reshetikhin-Turaev (RT) operator invariants of framed, oriented tangles \cite{RT1, Turaev:BOOK1}. We will denote by $F(T,V)$ the RT invariant of a tangle $T$ in which every component is colored with an object $V$ of the ribbon category $\CCgr$.

\def\dd{\boldsymbol{d}}
\def\LL{\boldsymbol{L}}
\subsection{Definition} Let $L$ be an oriented link in $S^3$. Give an arbitrary framing to $L$ and denote by $\LL$ the resulting framed, oriented link. Suppose $\LL$ is the closure of a $(1,1)$-tangle $\LL_o$. Let $\lla\in\hh^*$ be typical and suppose every component of $\LL_o$ is colored with $\Vla$. Since $\Vla$ is simple for typical $\lla$, the operator invariant $F(\LL_o, \Vla):\Vla\to \Vla$ is multiplication by a scalar $\lb \LL_o\rb$. We define $$\ADOgr(L,\lla):=q^{-(\lla,\lla-(r'-1)2\rho)w(\LL)} \lb \LL_o\rb$$
where $w(\LL)\in \Z$ is the writhe of $\LL$ and we claim this is an invariant of $L$. This is independent of the framing of $L$ chosen since $\theta_{\Vla}=q^{(\lla,\lla-(r'-1)2\rho)}\id_{\Vla}$. It is also independent of the place the link $L$ was cut open by the following argument: the category $\Proj(\CCgr)$ admits a two-sided modified trace $\tt$ \cite{GP:trace-projective-quantum-groups, Rupert:unrolled}, hence $\tt_{\Va}(F(\LL_o,\Vla))$ is a link invariant by \cite{GPV:traces-pivotal-cats}. Since $\Vla$ is simple $\tt_{\Va}(F(\LL_o,\Vla))=\lb \LL_o\rb \dd(\Va)$ (here $\dd$ denotes modified dimension) and since $\dd(\Vla)\neq 0$ and all components of $L$ have the same color, it follows that $\lb \LL_o\rb$ itself is a link invariant. Note that this argument is not needed if $L$ is a knot since in this case the $(1,1)$-tangle $\LL_o$ is uniquely determined by $\LL$.
\medskip

From the formula of the braiding, one can see that the coefficients of $c_{\Va,\Va}$ (in the tensor product basis of $\Va\ot\Va$) are in $\Q(q)[q^{(\lla,\lla)}, t_1^{\pm\frac{1}{2}},\dots, t_n^{\pm\frac{1}{2}}]$ where $t_i=q^{2\lla(H_i)}$. The framing factor in front of $\lb \LL_o\rb$ makes the $q^{(\lla,\lla)}$-factors to cancel out, so that $$\ADOgr(L)\in \Q(q)[t_1^{\pm\frac{1}{2}},\dots, t_n^{\pm\frac{1}{2}}].$$
 In fact the coefficients are in $\Z[q^2][t_1^{\pm\frac{1}{2}},\dots, t_n^{\pm\frac{1}{2}}]$, but we won't need this. We can also think of $\ADOgr(L)$ as an element in $\Z[q^2][\frac{1}{2}L_R]$ where $\tt^{\a}=t_1^{a_1}\dots t_n^{a_n}$ if $\a=\sum_{i=1}^na_i\a_i\in \frac{1}{2} L_R$. Then $\deg_{\tt}\ADOgr(L)$ is meant to be the difference of the highest $\frac{1}{2}L_R$-power of $\tt$ (with respect to the usual order of $L_R$) minus the lowest power. By the symmetry property below, if $\a\in \frac{1}{2}L_R$ is the highest power with non-zero coefficient, then $-\a$ is the lowest power so that $\det_{\tt}\ADOgr(L)=2\a\in L_R$.
 \medskip

 We will also denote by $\Pgr(L)$ the 1-variable polynomial obtained by equating all the variables of $\ADOgr$: $$\Pgr(L,t):=\ADOgr(L)|_{t_1=\dots=t_n=t}\in \Z[q^2][t^{\pm\frac{1}{2}}].$$

\subsection{Symmetry of ADO} 
Let $\o$ be the algebra automorphism of $\UHr$ defined by $\o(E_i)=F_i, \o(F_i)=E_i, \o(K_i)=K_i^{-1}, \o(H_i)=-H_i$, then $\Vla^*$ and $\o(\Vla)$ have the same character, so they are isomorphic for typical $\lla$ \cite{GP:trace-projective-quantum-groups, Rupert:unrolled}. 

\begin{proposition}
    $ADO_{\gg,r}(L)$ is insensitive to reversing the orientation of all components of $L$.
\end{proposition}
\begin{proof}
Let $\ov{L}$ be the link $L$ with the orientation reversed on all components. As usual, we have $F(\ov{L}_o, \Vla)=F(L_o, \Vla^*)$. But $\Vla^*\cong \o(\Vla)$ as mentioned above. Now, rotating a braid by 180 degress around a vertical axis does not change any invariant. But it clearly exchanges the $E$'s with $F's$ on all crossings, and the pivots $g$ to $g^{-1}$'s, hence it is the same as applying $\o$. This shows that $F(L_o,\o(\Vla))=F(L_o,\Vla)$, hence $F(\ov{L}_o,\Vla)=F(L_o,\Vla)$ from which $\ADOgr(\ov{L})=\ADOgr(L)$ follows.
\end{proof}

\begin{proposition}
\label{prop: symmetry of ADO}
The scalars $\ADOgr(L,\lla)$ satisfy $$\ADOgr(L,\lla)=\ADOgr(L,-\lla+(r'-1)2\rho).$$
In particular, if $\gg$ is simply-laced then $\ADOgr(L)$ is invariant under the transformation of variables $t_i^{\frac{1}{2}}\mapsto t_i^{-\frac{1}{2}}q^{-2}$ for all $i=1,\dots,n$.
\end{proposition}

Thus, if $\gg$ is simply-laced and if we set $$\Ngr(L,x_1,\dots,x_n):=\ADOgr(L,q^{-2}x_1,\dots,q^{-2}x_n)$$ then we get an honest symmetric polynomial, that is $\Ngr(x_1,\dots,x_n)=\Ngr(x_1^{-1},\dots,x_n^{-1})$. Similarly, the one-variable polynomial $\Pgr(L,t)$ is symmetric in the variable $x=q^2t$.

\begin{proof}
The modules $\Vla^*$ and $V_{-\lla+(r'-1)2\rho}$ have the same character, so they are isomorphic for typical $\lla$ \cite{GP:trace-projective-quantum-groups, Rupert:unrolled}. Thus, we have $$F(L_o,\Va)=F(\ov{L}_o,\Va)=F(L_o,\Va^*)=F(L_o,V_{-\lla+(r'-1)2\rho})$$ where the first equality follows from the previous proposition. This proves the first statement. Now, it is well-known that $2\rho(H_i)=(2\rho,\a_i)=(\a_i,\a_i)$ and this is equal to 2 for every $i$ in the simply-laced case. Thus, the right hand side is $\ADOgr$ with variables specialized to $q^{-\la_i+(r'-1)2\rho(H_i)}=q^{-\la_i-2}$ for every $i$ (here $\la_i=\lla(H_i)$), that is $$\ADOgr(q^{\la_1},\dots, q^{\la_n})=\ADOgr(q^{-\la_1-2},\dots, q^{-\la_n-2})$$
   from which the result follows.
\end{proof}

\def\uqb{\mathfrak{u}_q(\mathfrak{b})}
\subsection{The genus bound} In \cite{LNV:genus-unrolled} (or \cite{LNV:genus} for the $\sl_2$ case) we proved that $$\deg_{\tt}\ADOgr(L)\leq (r'-1)2\rho\cdot (2g(L)+s-1)$$ 
as elements of $L_R$, where $s$ is the number of components of $L$ and $g(L)$ the minimal genus of a connected Seifert surface of $L$. Since equating all the variables transforms each $\tt^{\a},\a\in \frac{1}{2}L_R$ into $t^{(\a,\rho)}$ we also get $$\deg_t\Pgr(L,t)\leq (r'-1)(2\rho,\rho)(2g(L)+s-1).$$
Note that $(r'-1)(2\rho,\rho)$ is the maximal $\N$-degree of the small Borel $\uqb$, i.e., the subalgebra of $\uqg$ generated by the $K_i^{\pm 1}, E_i$ for $i=1,\dots,n$ (where each $E_i$ has degree one), so the last bound is exactly the genus bound for 1-variable polynomials of \cite{LNV:genus} (extended to links).

\begin{proposition}
\label{prop: top of ADOgr is top of Pgr}
    The coefficient of $\tt^{(r'-1)\rho \cdot (2g(L)+s-1)}$ in $\ADOgr(L)$ is the same as the coefficient of $t^{(r'-1)(\rho,\rho)(2g(L)+s-1)}$ in $\Pgr(L)$.
\end{proposition}
\begin{proof}
    This is simply because $\a<\beta$ in $\frac{1}{2}L_R$ implies $(\a,\rho)<(\beta,\rho)$. Since the $\ADOgr$ bound holds for the $L_R$-degree, the only term of $\ADOgr$ that contributes to $t^{(r'-1)(\rho,\rho)(2g(L)+s-1)}$ after specialization is the term in degree $\tt^{(r'-1)\rho(2g(L)+s-1)}$.  %the $\ADOgr$ bound holds for $L_R$-degree. Thus, any other $\tt^{\alpha}$ appearing with non-zero coefficient in $\ADOgr$ must satisfy $\a<(r'-1)\rho (2g(L)+s-1)$ in $L_R$ so after setting $t_1=\dots=t_n=t$, the power of $t$ coming from $\tt^{\a}$ is $<$ the one coming from $\tt^{(r'-1)\rho(2g(L)+s-1)}$ which is $t^{(r'-1)(\rho,\rho)(2g(L)+s-1)}$.
\end{proof}

Write $2\rho=\sum_{i=1}^nk_i\a_i$ and $N=2g(L)+s-1$. Then, if the top term in $\ADOgr(L)$ is $a\cdot \tt^{(r'-1)\rho \cdot N}=a\prod_it_i^{\frac{(r'-1)}{2}Nk_i}$, then by the above symmetry property the least term is $$a\prod_it_i^{-\frac{r'-1}{2}Nk_i}q^{-2(r'-1)Nk_i}=a\prod_it_i^{-\frac{r'-1}{2}Nk_i}\cdot q^{2N\sum_{i=1}^nk_i}.$$

Noting that $(2\rho,2\rho)=\sum_i (2\rho,k_i\a_i)=2\sum k_i$ in the simply-laced case, we have thus shown that 

\begin{proposition}
    In type ADE, the least coefficient of $\ADOgr(L)$ (or $\Pgr(L)$) is related to the top coefficient by
$$\leastgr(L)=\topgr(L)\cdot q^{(2g(L)+s-1)(2\rho,2\rho)}.$$
\end{proposition}

\subsection{Hopf links}\label{subs: Hopf links computation} The invariant of the Hopf links are readily computed from the definition (as done in \cite[Prop. 45]{GP:nonrestricted}). Indeed, write an open Hopf link as a (partial) right closure of the double braiding $c\circ c$ with $c=c_{\Va,\Va}$. Clearly $c(v_0\ot v_I)=q^{(\la,\la-d(I))}v_I\ot v_0$ since $v_0$ is a highest weight vector. One has $$c(v_I\ot v_0)=q^{(\la,\la-d(I))}v_0\ot v_I+ \text{ terms of the form } c_{A,B} \cdot v_A\ot v_B$$
where $d(B)>d(I)$ and $c_{A,B}$ is some coefficient. These higher degree terms do not contribute to the partial right trace, so we can discard them. The pivot $g=K_{2\rho}^{1-r'}$, which acts on $v_I$, contributes an extra $q^{(1-r')(2\rho,\la-d(I))}$ factor. Since the diagram used has writhe 2, we obtain $$\ADOgr(H_+)=q^{-2(\la,\la+(1-r')2\rho)}\sum_{I:\De^+\to \{0,\dots,r'-1\}}q^{2(\la,\la-d(I))+(1-r')(2\rho,\la-d(I))}.$$
The top term in this expression is achieved when $I=0$, for which we see $\topgr(H_+)=1$. The least term is achieved at $I=\Imax$ and we find $$\leastgr(H_+)=q^{-(1-r')(2\rho,(r'-1)2\rho)}=q^{(2\rho,2\rho)}$$
where we use that $\gg$ is simply-laced in the last equality. Similarly, one finds $$\topgr(H_-)=q^{-(2\rho,2\rho)}, \ \leastgr(H_-)=1.$$

\section{Algebraic lemmas}
\label{sect: algebraic lemmas}

In this section, we gather the purely algebraic lemmas needed for the proof of Theorem \ref{theorem: Murasugi sum multiplicativity of ADO}. These concern the coefficients of the action of $\Uq(\gg)$ on $\Vla$ and $X=\Va\ot \Va^*$ as well as the action of the cointegral $\La$ of $\uqL$ (or $\uqg$ if ribbon) and the element $F^{\Imax}E^{\Imax}$ on the module $X$. We will suppose $\lla\equiv \la$ is diagonal and typical and denote $V_{\lla}$ by $V_{\la}$. Set $t^{\frac{1}{2}}=q^{\la}$.

\subsection{Algebraic bounds } The module $V_{\lambda}$ has basis the vectors $v_I:=F^Iv_0$ where $I:\Dep\to \{0,\dots,r'-1\}$. We will consider $V^*_{\la}$ with the dual basis $v'_J$, that is, $\lb v'_I,v_J\rb=\d_{I,J}$ and $X=V_{\lambda}\ot V_{\lambda}^*$ with the tensor product basis. We refer to these bases (and tensor products of these) as {\em standard bases}.

\begin{lemma}\cite{LNV:genus-unrolled}
\label{lemma: Ei coefficients action on Va}
    In the standard basis of $V_{\lambda}$, $E_i$ acts by a matrix with coefficients of the form $a(q)t^{\frac{1}{2}}+b(q)t^{-\frac{1}{2}}$ with $a(q),b(q)\in \Q(q)$.  In the dual basis of $V_{\lambda}^*$, $E_i$ acts by a matrix with coefficients of the form $c(q)+d(q)t^{-1}$, $c(q),d(q)\in \Q(q)$.
\end{lemma}

\begin{lemma}\cite{LNV:genus-unrolled}
  In the standard basis of $X$, and for each $i=1,\dots,n$,  $E_i$ acts by a matrix whose coefficients have the form $c(q)+d(q)t^{-1}$ with $c(q),d(q)\in\Q(q)$ while both $F_i, K_i$ act by matrices with $\Q(q)$-coefficients.
\end{lemma}

From these we get the following algebraic degree bound (the algebraic counterpart of the genus bound of \cite{LNV:genus, LNV:genus-unrolled}):

\begin{lemma}\cite{LNV:genus-unrolled}
\label{lemma: algebraic bounds}
    A PBW basis vector $K^AF^IE^J$  acts on the standard basis of $X$ by a matrix whose coefficients are (plain, non-Laurent) polynomials in $t^{-1}$ of degree $\leq |J|$. In particular, all coefficients have degree $\leq (r'-1)(2\rho,\rho)$ in $t^{-1}$.
\end{lemma}

Recall that $|J|$ is the $\N$-degree of $d(J)=\sum_{\a\in\De^+} J(\a)\a\in L_R$ whenever all simple positive roots are set to have degree one.

\subsection{Action of the cointegral on $X$} Let $\La$ be a non-zero cointegral of $\uqL$. It is well-known that this is a 2-sided cointegral (i.e. $\uqL$ is unimodular, see e.g. \cite[Theorem 7.3]{BBGa:modified} when $L=L_R$, for other $L$ the proof is similar), hence it is a central element so left multiplication by $\La$ is a $\uqL$-morphism on any $\uqL$-module. As is well-known, an explicit cointegral is given by $\La=\frac{1}{|G|}\sum_{g\in G}g\cdot F^{\Imax}E^{\Imax}$ where $G$ is the group of group-likes of $\uqL$ and $\Imax:\De^+\to \N$ is given by $\Imax(\a)=r'-1$ for all $\a\in\De^+$.

\begin{lemma}
\label{lemma: CGPT's lemma for cointegral}
For any typical $\lla$, there is a constant $\kappa_{\La}(\lla)\in \Q(q,q^{\pm \la_1},\dots,q^{\pm\la_n})$ such that $$\La|_{\Va\ot\Vad}=\kappa_{\La}(\la)\cdot  \lcoev_{\Va}\circ \rev_{\Va}.$$
Moreover, if $\La=\frac{1}{|G|}\sum_{g\in G}g\cdot F^{\Imax}E^{\Imax}$, the scalar $\kappa_{\La}(\la)$ coincides with the coefficient of the highest weight vector $v_0$ in $E^{I_{max}}F^{I_{max}}v_0$.
\end{lemma}

\begin{proof}
The first part is essentially a special case of \cite[Lemma 6.2]{CGPT} as we now explain. For every projective $\uqL$-module $P$, it is shown in \cite[Lemma 6.2]{CGPT} that there is a non-degenerate pairing\footnote{The above non-degenerate pairing is the so called ``modified trace'' of Geer, Patureau-Mirand and Turaev \cite{GPT:modified}, but we won't need this.} $$\Hom_{\uqL}(\unit,P)\otimes \Hom_{\uqL}(P,\unit)\to \C$$
such that if $(x_i)$ is a basis of $\Hom_{\uqL}(\unit,P)$ and $(y_i)$ is the dual basis of $\Hom_{\uqL}(P,\unit)$ under this pairing, then the cointegral $\La$ acts on $P$ as 
\begin{equation}
\label{eq: cointegral on projective}
    \La|_P=\c\sum_i x_i\circ y_i
\end{equation}
where $\c$ is some scalar (which can be chosen to be 1 if appropriate normalizations for the pairing and $\La$ are chosen). Now if we let $P=\Vla\ot \Vla^*$ (which is projective since $\Vla$ is), then both $\Hom_{\CCgr}(\unit,P), \Hom_{\CCgr}(P,\unit)$ are 1-dimensional (since both are isomorphic to $\End_{\CCgr}(\Vla)$ and $\Vla$ is simple), hence $\lcoev_{\Va}$ and $\rev_{\Va}$ (which are non-zero) are bases of each space. It follows that $\lcoev_{\Va}$ and $\rev_{\Va}$ are dual bases up to a scalar, so that the lemma follows by plugging these bases in (\ref{eq: cointegral on projective}).
\medskip

Now fix $\La=\frac{1}{|G|}\sum_{g\in G}g\cdot F^{\Imax}E^{\Imax}$. The scalar $\kappa_{\La}(\la)$ is computed explicitly by taking left-traces on both sides of the above equation: this gives
 $$\tr_{\Va}(K_{2\rho}^{r'-1}\La_{(1)})\La_{(2)}|_{\Va^*}=\kappa(\la)\id_{\Va^*}$$
Applying this to the highest weight vector $v'_{min}\in \Vla^*$ (the dual of the lowest weight vector $v_{min}=F^{I_{max}}v_0\in \Va$) it follows that one cannot have a power of some $E_i$ in $\La_{(2)}$, hence all positive degree is in $\La_{(1)}$. But then all negative degree also has to be in $\La_{(1)}$, otherwise it does not contributes to the trace. Thus, the only term in $\De(\La)$ that contributes to $\kappa(\la)$ is $$\frac{1}{|G|}\sum_{g\in G} g\ot g\cdot F^{\Imax}\ot 1\cdot E^{\Imax}\ot K^{d(\Imax)}.$$

Noting that $K^{d(\Imax)}=K_{2\rho}^{r'-1}$, we get
\begin{align*}
   \kappa(\la)v'_{min}= \tr_{\Va}(K_{2\rho}^{r'-1}\La_{(1)})\La_{(2)}v'_{min}&=\frac{1}{|G|}\sum_{g\in G}gK^{d(\Imax)}v'_{min}\tr_{\Va}(K^{d(\Imax)}gF^{\Im}E^{\Im})\\
    &=\frac{1}{|G|}\sum_{g\in G}gv'_{min}\tr_{\Va}(gF^{\Im}E^{\Im})
\end{align*}
Only the vector $v'_{min}$ contributes to the trace inside the previous formula (because $E^{\Imax}$ is non-zero only on $v_{min}$). Thus, if $F^{\Imax}E^{\Imax}v_{min}=\wt{\kappa}(\la)v_{min}$ then
\begin{align*}
    \kappa(\la)v'_{min}=\wt{\kappa}(\la)v'_{min}
\end{align*}
since the $g$-contributions cancel out for each $g\in G$. Thus $\wt{\kappa}(\la)=\kappa(\la)$. Writing $v_{min}=F^{\Imax}v_0$, it is easy to see that $\wt{\kappa}(\la)$ is also the coefficient of $v_0$ in $E^{\Imax}F^{\Imax}v_0$ as desired.
\end{proof}

For $\sl_2$ one has $E^{\Imax}F^{\Imax}v_{0}=E^{p-1}F^{p-1}v_0=\prod_{s=0}^{p-2}[\ov{\a}-s][p-1]!v_0$ (where $\ov{\a}=\la(H)$), explaining the formula of our previous paper \cite{LNV:fibred-sl2}.  
\medskip

From the description of $\kappa_{\La}(\la)$ in the above lemma it is easy to see that it is a Laurent polynomial in the $q^{\la_i}$ with maximal power $q^{(\la,d(\Imax))}=q^{(r'-1)(\la,2\rho)}$ or equivalently, it has maximal power $\tt^{(r'-1)\rho}$ if $t_i=q^{2\la(H_i)}$. This is because, for each $i=1,\dots,n$, $E_i(F^{\Imax}v_0)$ has coefficients of the form $a(q)q^{-\la_i}+b(q)q^{\la_i}$ for some $a(q),b(q)\in \Q(q)$ (Lemma \ref{lemma: Ei coefficients action on Va}) hence $E^{\Imax}F^{\Imax}v_0$ has coefficients of the desired form.

\begin{lemma}
\label{lemma: symmetry of kappa}
    The scalars $\kappa_{\La}(\la)$ satisfy the following symmetry property: $$\kappa_{\La}(\la)=\kappa_{\La}(-\la+(r'-1)\cdot 2\rho)$$
\end{lemma}
\begin{proof}
Write $\kappa(\la)=\kappa_{\La}(\la)$. Using that $\Vla^*\cong V_{-\la+(r'-1)\cdot 2\rho}$ and the previous lemma, we get that %$$\La|_{X'}=\kappa(-\la+(r'-1)\cdot 2\rho)\cdot\rcoev_{\Va}\circ \lev_{\Va}$$
      \begin{figure}[H]
    \includegraphics[width=6cm]{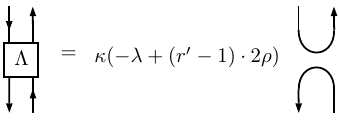}
\end{figure}
   By the definition of $\kappa(\la)$ it thus suffices to prove that 
\begin{figure}[H]
    \includegraphics[width=2.7cm]{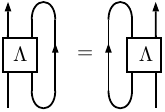}
\end{figure}

To prove this, first note that:
    \begin{figure}[H]
    \includegraphics[width=2.3cm]{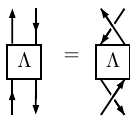}
\end{figure}
This follows simply because the braiding is a morphism of $\Uqg$-modules. Taking right traces we obtain
  \begin{figure}[H]
    \includegraphics[width=5.2cm]{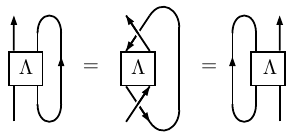}
\end{figure}
where the second equality follows from a simple isotopy. This proves the lemma.
\end{proof}

\subsection{Lemmas on the action of $F^{\Imax}E^{\Imax}$} Recall that $\Imax:\De^+\to \{0,\dots,r'-1\}$ is defined by $\Imax(\b)=r'-1$ for all $\b\in\De^+$. The top coefficient of ADO will be controlled by this top PBW vector, from which we need a couple lemmas about its action on the module $X$. The following generalizes \cite[Lemma 4.1]{LNV:fibred-sl2}.

\begin{lemma}
Suppose $d(I)=\sum_{i=1}^n a_i\a_i$ and $d(J)=\sum_{i=1}^n b_i\a_i$. Then
    \begin{itemize}
        \item $E^{\Im}(v_I\ot v'_J)=0$ if $a_i<b_i$ for some $i$.
        \item $F^{\Im}(v_I\ot v'_J)=0$ if $a_i>b_i$ for some $i$.
    \end{itemize}
\end{lemma}
\begin{proof}
    The vector $v_I\ot v'_J$ has weight $-d(I)+d(J)$ so $E^{\Im}(v_I\ot v'_J)$ has weight $-d(I)+d(J)+d(\Im)$. Since $-a_i+b_i>0$, the $\a_i$-component of this vector is greater than that of $d(\Im)$. Since there is no non-zero vector of such weight in $\Vla\ot \Vla^*$, this vector must be zero. The second statement is proved similarly.
\end{proof}

\begin{lemma}
\label{lemma: commuting F^I with E^I}
For all $f\in U^-, e\in U^+$ one has
    $$[f,e]= \text{ sum of terms with lower } L_R-\text{degree than } e.$$
\end{lemma}
\begin{proof}
    Obviously true for $f=F_i,e=E_j$. Using that $$[ff',e]=f[f',e]+[f,e]f'$$
    it follows by induction that $[f,E_j]$ is in $U^{\leq 0}$ for all $f\in U^-$. Similarly, using that $$[f,ee']=[f,e]e'+e[f,e']$$
    an another induction (in which the previous case is the base case), the lemma follows.
\end{proof}

\begin{lemma}
  If $d(I)\neq d(J)$, then $F^{\Im}E^{\Im}(v_I\ot v'_J)$ is a vector of $X$ for which, when expressed in the standard basis, all coefficients are polynomials in $t^{-1}$ of degree $<(r'-1)(2\rho,\rho)$.
   
\end{lemma}
\begin{proof}
    If for some $i$, the $\a_i$-component of $d(I)$ is $<$ than that of $d(J)$, then $F^{\Im}E^{\Im}(v_I\ot v'_J)=0$ by the first part of the previous lemma. If on the other hand, some $\a_i$-component of $d(I)$ is $>$ than that of $d(J)$, then one can commute $$F^{\Im}E^{\Im}=E^{\Im}F^{\Im}+\sum_{A,B,C}c_{A,B,C}K^AF^BE^C$$
    with $d(B),d(C)<d(\Im)$ and $c_{A,B,C}\in \Q(q)$ by Lemma \ref{lemma: commuting F^I with E^I}. But then $$F^{\Im}E^{\Im}(v_I\ot v'_J)=\sum_{A,B,C}c_{A,B,C}K^AF^BE^C(v_I\ot v'_J)$$
    by the second part of the previous lemma. Finally, since $d(C)<d(\Im)=(r'-1)2\rho$ for each $C$ on the RHS, we get the desired result by Proposition \ref{lemma: algebraic bounds}.
\end{proof}

Recall that we choose the cointegral $\La=\La_G F^{I_{max}}E^{I_{max}}=F^{I_{max}}E^{I_{max}}\La_G$ where $\La_G=\frac{1}{|G|}\sum_{x\in G}x$ and $G$ is the group of group-likes of $\uqL$.

\begin{lemma}
\label{lemma: change FEmax by cointegral}
    As operators of $X=\Vla\ot \Vla^*$ one has $$F^{I_{max}}E^{I_{max}}|_{X}=\La|_{X} + B$$
    where $B$ is an operator which, when expressed as a matrix in the standard basis of $X$, all coefficients are polynomials in $t^{-1}$ of degree $<(r'-1)(2\rho,\rho)$.
\end{lemma}
\begin{proof}
Let $B:X\to X$ be the action of $F^{\Imax}E^{\Imax}-\La$, we will check that $B(v_I\ot v'_J)$ has coefficients of the desired degree. Indeed, if $d(I)\neq d(J)$, $F^{I_{max}}E^{I_{max}}(v_I\ot v'_J)$ has coefficients of the desired degree by the previous lemma while $\La(v_I\ot v'_J)=0$ if $I\neq J$ (in particular if $d(I)\neq d(J)$) by Lemma \ref{lemma: CGPT's lemma for cointegral}, hence the same holds for $B(v_I\ot v'_J)$. On the other hand, $\frac{1}{|G|}\sum_AK^A$ acts as the identity on weight zero vectors, hence if $d(I)=d(J)$ one has $$\La(v_I\ot v_J)=F^{I_{max}}E^{I_{max}}(v_I\ot v_J)$$
so that $B(v_I\ot v_J)=0$ in such a case. These two cases imply the lemma.
\end{proof}

\section{Proof of main theorems}
\label{sect: proofs main thms}
\def\FC{F_{\CCgr}}
\def\TS{T_{\S}}

\def\FF{\overline{F}}

\def\JJ{\mathcal{J}_X}
\def\XX{X^{\ov{\e}}}
\subsection{Seifert surface formula}\label{subs: Seifert formula}
Let $L$ be an unframed, oriented link and let $\S$ be a connected Seifert surface for $L$. Then $\S$ can be obtained from a single disk by attaching, say, $N$ 1-handles. After an isotopy, we can suppose $\S$ is obtained by taking an $N$-component bottom tangle $\TS$\footnote{This is almost the terminology of Habiro \cite{Habiro:bottom-tangles}, except that we do not require any condition on the placement of the endpoints of the components of $T$.}, i.e. a framed tangle in $\R^2\t[0,1]$ with all endpoints on $\R^2\t\{1\}$, and then attaching a disk on top, see Figure \ref{figure:Seifert}. We will consider $L$ as a framed link with the framing coming from $\S$, thus $w(L)=0$ and $\ADOgr(L)=\ADOgr'(L)$. Let $L_o$ be the framed $(1,1)$-tangle obtained by opening $L$ at $p$ as indicated in the same figure.
\begin{figure}[H]
    \centering
    \includegraphics[width=8.5cm]{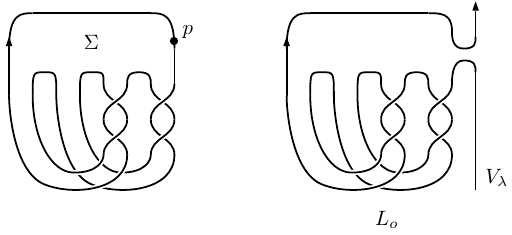}
    \label{figure:Seifert}
    \caption{On the left is the Seifert surface $\S$ of $L$. On the right, we open $L=\p\S$ at $p$ to obtain a $(1,1)$-tangle $L_o$. This is colored with $\Vla$ so the figure has to be interpreted as the RT invariant $F(L_o,\Vla):\Vla\to \Vla$.}
\end{figure}

\subsubsection{$F(L_o,\Vla)$ from $F(\TS,X)$} We will express the invariant $F(L_o,\Vla)$ (and hence $\ADOgr(L)$) in terms of $F(\TS,X)$ as follows. Suppose each component of $\TS$ is oriented to the right, that is, for any component $T_i$ of $\TS$ the orientation of $T_i$ is downwards at the leftmost endpoint of $T_i$ on $\R^2\t\{1\}$. Let $\e_1,\dots,\e_{2N} \in \{\pm 1\}$ be the signs at the endpoints of $\TS$ from left to right, where $\e_i=+1$ (resp. $\e_i=-1$) indicates upwards orientation  (resp. downwards). Let $X^{+1}=X$ and $X^{-1}=X^*$ and denote $\XX=X^{\e_1}\ot\dots\ot X^{\e_{2N}}$. Then the Reshetikhin-Turaev invariant of $\TS$ colored with $X$ has the form $$F(\TS,X):\C\to  \XX=X^{\e_1}\ot\dots\ot X^{\e_{2N}}.$$
Let  $\JJ:\XX\to X^{\ot 2N}$ be the isomorphism that applies $\iota_X$ to each $X^*$-tensor factor and $\id_X$ on each $X$-tensor factor. Let 
\begin{align}
    \label{eq: def of F bar}
    \FF(\TS,X)=\JJ\circ F(\TS,X):\C\to X^{\ot 2N}
\end{align}

 Then the invariant $F(L_o,\Vla):\Vla\to \Vla$ is computed from $\FF(\TS,X)$ by
$$F(L_o,\Vla)=(\id_{\Vla}\ot \lev_{\Vla}^{\ot 2N})\circ(\FF(\TS,X)\ot\id_{\Vla}).$$
Here is a picture in the case of Figure \ref{figure:Seifert} (where $N=2$):
\begin{figure}[H]
    \centering
    \includegraphics[width=10cm]{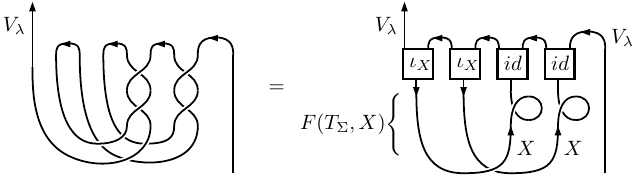}
\end{figure}
\def\ZZ{Z^{ev}}

\subsubsection{$F(\TS,X)$ from universal invariants} %Suppose for a while that $(r,\det(A))=1$ so that $\uqg$ ribbon (or that $\gg=\sl_2$ and $\uqLW(\sl_2)$ is used in what follows). 
Let $L$ be a lattice between $L_R$ and $L_W$ as in Lemma \ref{lemma: braiding on cXX expressed as R-matrix on some small uqgL}. Since $X,X^*\in \CC_{[0]}$ and the invariant $F(\TS,X)$ is built by composing/tensoring $c_{X,X},c_{X,X^*},c_{X^*,X},c_{X^*,X^*}$'s and (co)evaluations, which can all be expressed as left multiplication by elements in $\uqL$ by Lemma \ref{lemma: braiding on cXX expressed as R-matrix on some small uqgL}, it follows that  
there is an element $$Z_{T}\in \uqL^{\ot N}$$ that determines $F(\TS,X)$ by left multiplication on $X^{\ot N}$ (as explained below). Whenever $\uqL$ is ribbon in such a way that the forgetful functor $F:\CC_{[0]}\to \Rep(\uqL)$ is ribbon (e.g. if $(r,\det(A))=1$ or if $\gg=\sl_2$ and $L=L_W$), this element is the universal $\uqL$-invariant of the bottom tangle $\TS$ as in \cite{Habiro:bottom-tangles} or \cite[Chapter 4]{Ohtsuki:BOOK}.

\medskip

The element $Z_T$ determines $F(\TS,X)$ as follows: let $\ZZ_T\in \uqL^{\ot 2N}$ be the element obtained from $Z_T$ by inserting the unit $1_{\uqL}$ in all odd tensor factors. Let $L_{\ZZ_T}$ be the endomorphism of $(X^*\ot X)^{\ot N}$ that is left multiplication by $\ZZ_T$, so the units $1_{\uqL}$ in $Z_T^{ev}$ act on the $X^*$-factors and each tensor factor of $Z_T$ acts on the corresponding $X$-factor. Then
\begin{align*}
    F(\TS,X)= P_{X,X^*}\circ L_{\ZZ_T}\circ  \coev_X^{\ot N}
\end{align*}

\noindent where $\coev_X$ denotes the usual right coevaluation of vector spaces. Here $P_{X,X^*}$ is a permutation isomorphism (of $X$ and $X^*$ tensor factors) determined by the ordering of the components of $T$ and the way the endpoints of $T$ lie in $\R^2\t\{1\}$. Usually, in the above formula both $P$ and $\coev_X$ are those of vector spaces, so there is no $R$-matrix involved in $P_{X,X^*}$ or pivot in $\coev_X$ as all these Hopf algebra elements were already included in $Z_T$. However, we prefer to have $\rcoev_X$ in the bottom, so instead we will write

\begin{align}
\label{eq: RT of T in terms of universal}
    F(\TS,X)&= P\circ L_{(Z'_T)^{ev}}\circ  \rcoev_X^{\ot N}
\end{align}

\noindent where $Z'_T=Z_T\cdot g^{\ot N}\in \uqL^{\ot N}$, $g=K_{2\rho}^{1-r'}$, $(Z'_T)^{ev}$ is obtained from $Z'_T$ by inserting the unit in odd tensor factors and $L$ is left multiplication. Note that $\JJ\circ P_{X,X^*}=P_{X,X}\circ (\iota_X\ot \id_X)^{\ot N}$ where $P_{X,X}$ is the same permutation but applied to $X^{\ot N}$ so that
$\FF(\TS,X)$ we can be written as $$\FF(\TS,X)=P_{X,X}\circ (\iota_X\ot \id_X)^{\ot N}\circ L_{(Z'_T)^{ev}}\circ \rcoev^{\ot N}$$
or graphically 
\begin{figure}[H]
    \centering
    \includegraphics[width=8cm]{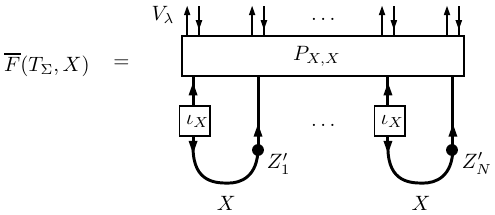}
\end{figure}
\noindent where we write $Z'_T=\sum Z'_1\ot\dots\ot Z'_N$. For the previous example:

\begin{figure}[H]
    \centering
    \includegraphics[width=10cm]{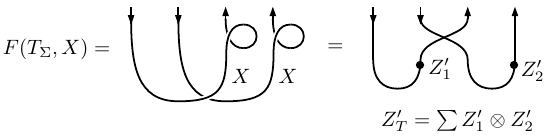}
\end{figure}

\medskip
To summarize this discussion:

\begin{proposition}
\label{prop: Seifert formula}
    If $\S$ is a connected Seifert surface of a link $L$ built from an $N$-component tangle $\TS$ as above, then the $\ADOgr$ invariant of $L$ is determined by $\TS$ as
    $$\ADOgr(L)\cdot\id_{\Vla}=(\id_{\Vla}\ot \lev_{\Vla}^{\ot 2N})\circ(\FF(\TS,X)\ot\id_{\Vla})$$
    where $\FF(\TS,X)=\JJ\circ F(\TS,X)$ is determined by left multiplication of the universal invariant $Z'_T$ acting on $X^{\ot N}$ as in (\ref{eq: RT of T in terms of universal}). 
\end{proposition}

\subsection{The main proposition and proof of Theorem \ref{theorem: top coefficient formula with integral}}

\def\maxdeg{\text{max-deg}}

Proposition \ref{prop: prop surface and coevaluations} below is the main ingredient behind the proof of our theorems. %In what follows we will have various expressions with coefficients that are Laurent polynomials in $\tt$ with powers in the root lattice $L_R$. Some coefficients will actually be polynomials in $t^{-1}$ (i.e. no positive powers of $t$) so we will always distinguish clearly between polynomials and Laurent polynomials. Recall that we denote $\tt^{\a}=t_1^{a_1}\dots t_n^{a_n}$ if $\a=\sum_{i=1}^na_i\a_i\in L_R$. 
By $\maxdeg_{t^{-1}}$ we refer to the maximal degree in $t^{-1}$ of a Laurent polynomial in $t$, for instance $\maxdeg_{t^{-1}}(t^{-2}+t)=2$. Similarly $\maxdeg_{t}(t^{-2}+t)=1$

\medskip
Recall that we set $\ov{\mu}(x):=\mu(\La_Gx)$ for all $x\in \uqL$ where $\mu:\uqL\to\C$ is the right integral normalized by $\mu(\La)=1$ where $\La=\La_GF^{\Imax}E^{\Imax}.$
\medskip

Finally, following the notation of the previous section, we will denote
\begin{align}
\label{eq: def of F(Z,X)}
    F(Z,X)&= P\circ L_{Z^{ev}}\circ  \rcoev_X^{\ot N}
\end{align}
for any $Z\in \Uq^{\ot N}$. Thus $F(\TS,X)=F(Z'_T,X)$, but note that $F(Z,X)$ is not a Reshetikhin-Turaev invariant in general. We will employ similar notation for $\FF(Z,X)$.

\begin{proposition}
\label{prop: prop surface and coevaluations}
    Let $\TS$ be an $N$-component framed, oriented bottom tangle. Then the Reshetikhin-Turaev operator invariant $\FF(\TS,X):\C\to X^{\ot 2N}$ satisfies
\begin{align}
\label{eq prop: FTX in terms of coev + B}
     \FF(\TS,X)=c(q,q^{\la})\cdot \lcoev_{\Va}^{\ot 2N}+B
\end{align}
\noindent or graphically
 \begin{figure}[H]
        \centering
        \includegraphics[width=13cm]{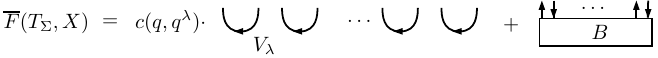}
    \end{figure}
   \noindent for some coefficient $c(q,q^{\la})\in \Q(q,q^{\pm\la})$ and where $B:\C\to X^{\ot 2N}$ has the following property: the coefficients of $B(1)\in X^{\ot 2N}$ in the standard basis are Laurent polynomials in $t^{-1}$ with $$\maxdeg_{t^{-1}}<(r'-1)(\rho,\rho) \cdot N \ \text{ and } \ \maxdeg_{t}\leq (r'-1)(\rho,\rho)\cdot N.$$ Moreover, the coefficient $c(q,q^{\la})$ is explicitly given by $$c(q,q^{\la})=\kappa(\la)^N\cdot\ov{\mu}^{\ot N}(Z_T)$$
        where $\kappa(\la)$ is as in Lemma \ref{lemma: CGPT's lemma for cointegral} and $Z_T\in\uqL^{\ot N}$ is the element defined in Subsection \ref{subs: Seifert formula}.
  
\end{proposition}

\begin{proof}

During the whole proof, by ``coefficients'' we mean coefficients in the standard basis of $X$ or $X^*$ (the dual basis) and tensor products of these. 
\medskip

\textbf{Step 1: Apply algebraic bounds to $Z'_T$.} We begin by applying (\ref{eq: RT of T in terms of universal}) to the LHS of (\ref{eq prop: FTX in terms of coev + B}):
\begin{align*}
     \FF(\TS,X)=\JJ\circ F(\TS,X), \ \ F(\TS,X)=P\circ L_{(Z'_T)^{ev}}\circ  \rcoev_X^{\ot N}
\end{align*}
Applying the algebraic bound of Lemma \ref{lemma: algebraic bounds} $N$ times we see that the coefficients of $L_{(Z'_T)^{ev}}\circ  \rcoev_X^{\ot N}$, and hence of $F(\TS,X)$, are polynomials in $t^{-1}$ with $\maxdeg_{t^{-1}}\leq (r'-1)(2\rho,\rho) N$. Note that since there are exactly $N$ repetitions of $\iota_X$ in $\JJ$, the operator $\JJ$ has coefficients in $t^{(r'-1)(\rho,\rho) N}\Q(q)$ (this has the effect of ``symmetrizing'' the powers of $t$ in $F(\TS,X)$). This proves that the coefficients of $\FF(\TS,X)$ are Laurent polynomials in $t^{\pm 1}$ with $\maxdeg_{t^{\pm 1}}\leq (r'-1)(\rho,\rho) N$ (which immediately implies the main result of \cite{LNV:genus-unrolled} by Prop. \ref{prop: Seifert formula}). Our purpose then is to pin down the coefficient of $t^{-(r'-1)(2\rho,\rho) N}$ in $F(\TS,X)$. The idea will be to expand $Z'_T$ in the (tensor product) PBW basis of $\uqL^{\ot N}$ and see which terms contribute (and how) to the power $t^{-(r'-1)(2\rho,\rho) N}$ in $F(\TS,X)$.
\medskip

\def\uq{\mathfrak{u}_q}

\textbf{Step 2: Expand $Z'_T$ in PBW and discard low $\uq^+$-powers.} So let's expand $Z'_T$ in the PBW basis.  One can write $$Z'_T=Z^{top} + Z^{low}$$
where $Z^{low}$ is the sum of all terms in the PBW-expansion of $Z'_T$ having some tensor factor of the form $K^AF^IE^J$ with $J<I_{max}$ and $Z^{top}$ is an element of the form $$Z^{top}=\bigotimes_{i=1}^N h_i(K,F)E^{I_{max}}$$
where $h_i(K,F)$ is a polynomial in $K_j^{\pm 1}$'s and $F_j$'s with $\Q(q)$-coefficients. Note that by Lemma \ref{lemma: algebraic bounds}, $Z^{low}$ does not contribute to the terms of highest degree in $t^{-1}$. More precisely, we have $$F(\TS,X)=F(Z'_T,X)=F(Z^{top},X)+F(Z^{low},X)$$
and all coefficients of $F(Z^{low},X)\in \XX$ are polynomials in $t^{-1}$  with $\maxdeg_{t^{-1}}<(r'-1)(2\rho,\rho)\cdot N$.
\medskip

\textbf{Step 3: Top $\uq^+$-power is cointegral.} It is thus enough to prove
\begin{align}
\label{eq proof prop: FTX in terms of coev + B'}
     \FF(Z^{top},X)=(\kappa(\la)^N\cdot\ov{\mu}^{\ot N}(Z_T))\cdot \lcoev_{\Va}^{\ot 2N}+B'
\end{align}
where $B'$ satisfies the same conditions as $B$. For this, we must simply understand the action of $Z^{top}$ on $X^{\ot N}$. The crucial point is to note that, since $Z'_T$ has degree zero (because the $R$-matrix does) and $F^I=0$ if $d(I)>d(I_{max})$, every term $F^IE^{I_{max}}$ appearing in $Z^{top}$ must satisfy $I=I_{max}$. Hence, $Z^{top}$ has the form $$Z^{top}=\bigotimes_{i=1}^N f_i(K) F^{I_{max}}E^{I_{max}}$$
where each $f_i(K)$ is in the Cartan subalgebra of $\uqL$. Now, by Lemmas \ref{lemma: change FEmax by cointegral} and \ref{lemma: CGPT's lemma for cointegral}, for each $i=1,\dots, N$ we have
\begin{align*}
    f_iF^{I_{max}}E^{I_{max}}|_{X}&=f_i\La|_{X} + B_i\\
    &=\e(f_i)\La|_{X}+B_i\\
    &=\e(f_i)\kappa(\la) \cdot\lcoev_{\Va}\circ\rev_{\Va}+B_i
\end{align*}
where $B_i$ is an operator with coefficients that are polynomials in $t^{-1}$ satisfying $\maxdeg_{t^{-1}}<(r'-1)(2\rho,\rho)$. Note that we used the cointegral property in the second equality. 

\medskip
\textbf{Step 4: The cointegral ``cuts''.} Writing $Z^{top}= Z^{top}_1\ot\dots\ot Z^{top}_N$ it follows that
\begin{figure}[H]
    \centering
    \includegraphics[width=14cm]{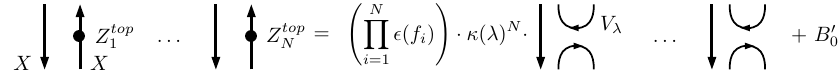}
\end{figure}

\noindent where $B'_0$ is an operator with coefficients polynomials in $t^{-1}$ with $\maxdeg_{t^{-1}}<(r'-1)(2\rho,\rho)\cdot N$ (and $\maxdeg_{t}\leq 0$). Letting $c(q,q^{\la})=\prod_{i=1}^N\e(f_i)\cdot \kappa(\la)^N$ and $$B'=P\circ (\iota_X\ot \id_X)^{\ot N}\circ B'_0\circ \rcoev_X^{\ot N}$$ we get
\begin{figure}[H]
    \centering
    \includegraphics[width=16cm]{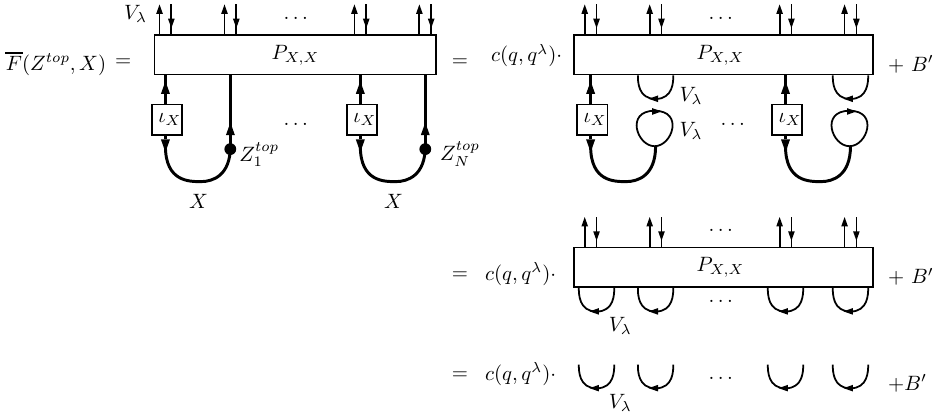}
\end{figure}

Here we used Lemma \ref{lemma: coev lemma} in the third equality while the fourth is obvious ($P_{X,X}$ is just a permutation isomorphism). Since $P$ and $\rcoev_X^{\ot N}$ have $\Q$-coefficients and $(\iota_X\ot \id_X)^{\ot N}$ has coefficients in $t^{(r'-1)(\rho,\rho) N}\Q(q)$, from what we said about $B'_0$ it follows that $B'$
has coefficients with $\maxdeg_{t^{-1}}<(r'-1)(\rho,\rho) N, \maxdeg_{t}\leq(r'-1)(\rho,\rho) N$. Summing up, we get
\begin{figure}[H]
    \includegraphics[width=13cm]{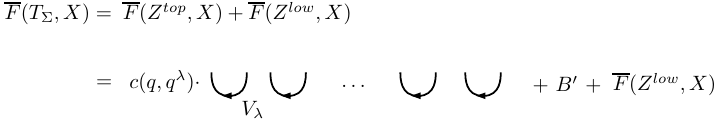}
\end{figure}
\begin{comment}
    
But then
\begin{align*}
    \FF(\TS,X)&=\FF(Z^{top},X)+\JJ\circ B_0\\
    &= c(q,q^{\la})\cdot \lcoev_{\Va}^{\ot 2N}+B
\end{align*}

\end{comment}
\noindent and $B:=B'+\FF(Z^{low},X)$ has coefficients with $\maxdeg_{t^{-1}}<(r'-1)(\rho,\rho) N, \maxdeg_{t}\leq(r'-1)(\rho,\rho) N$ since both $B'$ and $\FF(Z^{low},X)$ have this property. This proves (\ref{eq prop: FTX in terms of coev + B}).

\medskip

\textbf{Step 5: Pinning down $c(q,q^{\la})$.} We now prove that $c(q,q^{\la})=\kappa(\la)^N\ov{\mu}^{\ot N}(Z_T)$ or equivalently $\prod_{i=1}^N\e(f_i)=\ov{\mu}^{\ot N}(Z_T)$. To do this, note that $\ov{\mu}(K^AF^IE^J)=0$ unless $I=J=\Im$ (where it is =1), so that $$\ov{\mu}^{\ot N}(Z'_T)=\ov{\mu}^{\ot N}(Z^{top})=\prod_{i=1}^N\ov{\mu}(f_iF^{\Im}E^{\Im}).$$
But by definition of $\ov{\mu}$, one has $$\ov{\mu}(f_iF^{\Im}E^{\Im})=\mu(\La_Gf_iF^{\Im}E^{\Im})=\mu(f_i\La)=\mu(\e(f_i)\La))=\e(f_i)$$since $\mu(\La)=1$. Finally, note that $\ov{\mu}^{\ot N}(Z'_T)=\mu^{\ot N}(Z_T)$ since $\mu(xg)=\mu(gx)$ so that $\ov{\mu}(xg)=\ov{\mu}(x)$ for any group-like $g$ and $x\in \uqL$.
\medskip

\end{proof}

\begin{proof}[Proof of Theorem \ref{theorem: top coefficient formula with integral}]
Closing up successive strands in the above proposition (and using Prop. \ref{prop: Seifert formula}), we find 
%$$\ADOgr(L)\cdot\id_{\Va}=c(q,q^{\la})\cdot\id_{\Va}+(\id_{\Va}\ot \lev^{\ot 2N})(B\ot \id_{\Va})$$
 \begin{figure}[H]
        \centering
        \includegraphics[width=12cm]{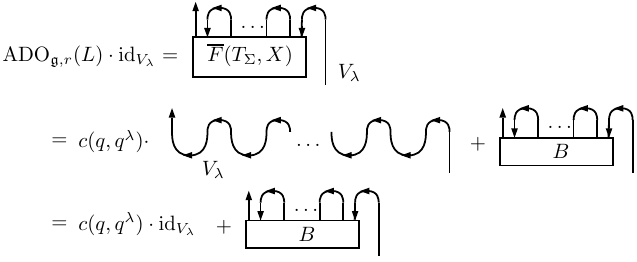}
      
    \end{figure}
But since $B$ has coefficients with $\maxdeg_{t^{-1}}<(r'-1)(\rho,\rho) N$ then it follows that
\begin{align*}
    \leastgr \ADOgr(\p\S)&=\leastgr c(q,q^{\la})=(\leastgr \kappa(\la)^N)\ov{\mu}^{\ot N}(Z_T).
\end{align*}
Since $\ADOgr$ and $\kappa(\la)$ satisfy the same symmetry property, it follows that \begin{align}
\label{eqtn: proof of Thm 3, with kappa lambda}
    \topgr \ADOgr(\p\S)&=\topgr c(q,q^{\la})=(\topgr \kappa(\la)^N)\ov{\mu}^{\ot N}(Z_T).
\end{align}
We can solve for $\topgr \kappa(\la)$ by plugging in $L=H_-$ (so that $Z_T=v$ and $N=1$ since $H_-$ bounds an annulus with a right-handed twist) and using $\topgr(H_-)=q^{-(2\rho,2\rho)}$ as computed in Subsection \ref{subs: Hopf links computation}. This gives $$q^{-(2\rho,2\rho)}=\topgr(H_-)=\ov{\mu}(v)\cdot\topgr\kappa(\la)$$
so that $$\topgr\kappa(\la)=\frac{q^{-(2\rho,2\rho)}}{\ov{\mu}(v)}.$$
Plugging the latter into (\ref{eqtn: proof of Thm 3, with kappa lambda}) establishes Theorem \ref{theorem: top coefficient formula with integral}.
\end{proof}

\begin{remark}
\label{remark: scalar ovmu on ribbon v}
If the cointegral of $\uqg$ used is $\La=\La_GF^{\Imax}E^{\Imax}$ and $\mu$ is such that $\mu(\La)=1$ then the scalar $\ov{\mu}(v)$ can be explicitly computed as $$\ov{\mu}(v)=\left(\frac{q^{1-r'}}{r'}(q-q^{-1})^{2(r'-1)}\right)^{|\De_+|}\cdot q^{-2(2\rho,2\rho)}.$$
    Indeed, by the formula of the $R$-matrix the ribbon element can be computed to be $$v=\sum_Iq^{-(d(I),d(I))}c_IF^IK_{2\rho}^{1-r'}E^I.$$
    Noting that after applying $\ov{\mu}$ only the term with $I=\Imax$ survives, using $d(\Imax)=(r'-1)2\rho$ and commuting $F^{\Imax}K_{2\rho}^{1-r'}=q^{-(r'-1)^2(2\rho,2\rho)}K_{2\rho}^{1-r'}F^{\Imax}$ and using $\La_GK_{2\rho}^{1-r'}=\La_G$ gives $$\ov{\mu}(v)=q^{-2(r'-1)^2(2\rho,2\rho)}c_{\Imax}\mu(\La_GF^{\Imax}E^{\Imax})=q^{-2(2\rho,2\rho)}c_{\Imax}.$$
    Plugging in (\ref{eq: formula for cImax}) we get the desired formula.
\end{remark}

Let $X'=\Vla^*\ot \Vla$. By turning Proposition \ref{prop: prop surface and coevaluations} upside-down we also obtain: 

\begin{proposition}
\label{prop: F(TS,X) as coev but upside-down}
    Let $T_{\S'}$ be a $N'$-component (framed, oriented) top tangle. Then the operator invariant $\FF(T_{\S'},X'):(X')^{\ot 2N'}\to \C$ satisfies $$\FF(T_{\S'},X')=c'(q,q^{\la})\cdot\lev_{\Va}^{\ot 2N'}+B'$$
    or graphically
    \begin{figure}[H]
        \centering
        \includegraphics[width=13cm]{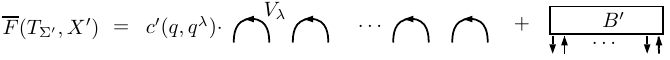}
      
    \end{figure}
   \noindent for some $c'(q,q^{\la})\in \Q(q,q^{\pm\la})$ and where $B'$ evaluates any standard basis vector of $(X')^{\ot 2N'}$ to a Laurent polynomial in $t$ with $\maxdeg_{t}<(r'-1)(\rho,\rho)\cdot N'$ and $\maxdeg_{t^{-1}}\leq (r'-1)(\rho,\rho)\cdot N'$.

\end{proposition}

The following lemma concerns the tensors $B$ and $B'$ of Propositions \ref{prop: prop surface and coevaluations} and \ref{prop: F(TS,X) as coev but upside-down}.

\begin{lemma}
\label{lemma: B has smaller deg after closing}
  The operators $$\cl(B):=(\id_{\Va}\ot\lev_{\Va}^{\ot 2N})\circ (B\ot \id_{\Va}):\Va\to \Va$$
  and
  $$\cl(B'):= (\id_{\Va}\ot B' )\circ (\lcoev_{\Va}^{\ot 2N'}\ot \id_{\Va}):\Va\to \Va$$
  have coefficients which are Laurent polynomials in $t$ satisfying $\maxdeg_{t^{\pm 1}}<(r'-1)(\rho,\rho) \cdot N$ and $\maxdeg_{t^{\pm 1}}< (r'-1)(\rho,\rho)\cdot N'$ respectively.
\end{lemma}
\begin{proof}
    %This follows from the symmetry of $\kappa(\la)$ (Lemma \ref{lemma: symmetry of kappa}) and what we just proved as follows. Applying $(\id_{\Va}\ot\lev_{\Va}^{\ot 2N})$ on both sides of (\ref{eq prop: FTX in terms of coev + B}) we obtain $$\ADOgr(L)\id_{\Va}=(\kappa(\la)^N\cdot \ov{\mu}^{\ot N}(Z_T))\id_{\Va}+(\id_{\Va}\ot\lev_{\Va}^{\ot 2N})\circ B.$$
As in the previous proof, after closing strands we get
$$\ADOgr(L)\cdot\id_{\Va}-(\kappa(\la)^N\ov{\mu}^{\ot N}(Z_T))\cdot\id_{\Va}=(\id_{\Va}\ot \lev^{\ot 2N})\circ(B\ot \id_{\Va}).$$
As before, the least term of $\ADOgr(L)$ coincides with that of $(\kappa(\la)^N\ov{\mu}^{\ot N}(Z_T)$. But since $\ADOgr$ and $\kappa(\la)$ satisfy the same symmetry (and $\ov{\mu}^{\ot N}(Z_T)\in\Q(q)$) by Lemma \ref{lemma: symmetry of kappa} and Proposition \ref{prop: symmetry of ADO}, it follows that they also have the same top term. This implies that the RHS cannot have a non-trivial coefficient in degree $(r'-1)(\rho,\rho)\cdot N$, as desired. The same argument applies to $\cl(B')$.
\end{proof}

\begin{remark}
    The above lemma would not have been necessary if the tensor $B$ had coefficients satisfying $\maxdeg_{t^{-1}}<(r'-1)(\rho,\rho) \cdot N$ and $\maxdeg_{t}< (r'-1)(\rho,\rho)\cdot N$, but we could only prove $\leq$ in the last inequality.%, as in Proposition \ref{eq prop: FTX in terms of coev + B}.
\end{remark}

\subsection{Proof of Theorem \ref{theorem: Murasugi sum multiplicativity of ADO}}
    Suppose $\S_1,\S_2$ are Murasugi-summed along a $2n$-gon to get the surface $\S$ and let $L=\p\S$. This $2n$-gon is topologically a disk $D$ which we can suppose lies in $\R^2\t \{0\}$ and, after sliding handles if necessary, we can suppose that both $\S_1$ and $\S_2$ are obtained by attaching 1-handles to this disk. Suppose $\S_1$ lies below $D$ and $\S_2$ lies above, let $T_{\S_1}, T_{\S_2}$ be the corresponding tangles and let $N,N'$ be their number of components respectively. See Figure \ref{figure: isotoping bands on Murasugi sum}.
\medskip

\begin{figure}
    \centering
    \includegraphics[width=11cm]{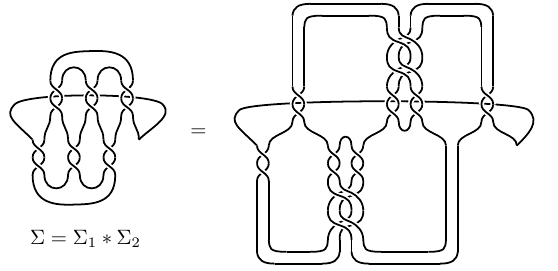}
    \caption{On the left there is an hexagonal Murasugi-sum of two pretzel surfaces of type $(-2,-2,-2)$. Isotoping the two rightmost bands, we get the surface on the right.}
    \label{figure: isotoping bands on Murasugi sum}
\end{figure}

    Then, to compute $\ADOgr(L)$ we can open $L$ at a point of $\p D$ away from the 1-handles so that $F(L_o):\Va\to\Va$ can be written as
    \begin{align*}
        F(L_o,\Va)=(\id_{\Va}\ot \FF(T_{\S_2}, X'))\circ I\circ (\FF(T_{\S_1},X)\ot \id_{\Va})
    \end{align*}
    where $X'=\Va^*\ot\Va$ and the operator $I:X^{\ot 2N}\ot\Va\to\Va\ot (X')^{2N'}$ is simply a tensor product of left evaluations and left coevaluations (and identity maps) as specified by the endpoints of $T_{\S_1}$ and $T_{\S_2}$ (see Figure \ref{figure: Murasugi sum in terms of bottom and top tangles}). 

\begin{figure}
    \centering
    \includegraphics[width=11cm]{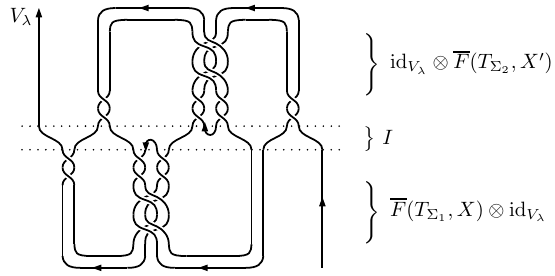}
    \caption{In this picture, we cut-open the link $L$ on the Murasugi-sum disk of the rightmost surface of Figure \ref{figure: isotoping bands on Murasugi sum}. The result is an expression of the operator invariant $F(L_o,\Va)$ in terms of the operator invariants $\FF(T_{\S_1},X)$ and $\FF(T_{\S_2},X')$ to which we can apply Propositions \ref{prop: prop surface and coevaluations} and \ref{prop: F(TS,X) as coev but upside-down}. In between the dotted lines we indicate the operator $I$, which is a tensor product of identity maps and left evaluations/coevaluations.}
    \label{figure: Murasugi sum in terms of bottom and top tangles}
\end{figure}

 \medskip

    Clearly the operator $I$ has only $\Q$-coefficients in the standard basis (it is a tensor product of identity maps and left (co)evaluations). Note also that $$(\id_{\Va}\ot \lev_{\Va}^{\ot 2N'})\circ I=\id_{\Va}\ot \lev_{\Va}^{\ot 2N}, \ I\circ (\lcoev_{\Va}^{\ot 2N}\ot \id_{\Va})=\lcoev_{\Va}^{\ot 2N'}\ot\id_{\Va}$$
    and $$(\id_{\Va}\ot \lev_{\Va}^{2N'})I(\lcoev^{\ot 2N}\ot\id_{\Va})=\id_{\Va}$$
    as it is clear from Figure \ref{figure: Murasugi sum in terms of bottom and top tangles}. Applying Propositions \ref{prop: prop surface and coevaluations} and \ref{prop: F(TS,X) as coev but upside-down} to $\FF(T_{\S_1},X)$ and $\FF(T_{\S_2},X')$, we get:
\begin{align*}
    F(L_o,\Va)&=(\id_{\Va}\ot c'\cdot\lev_{\Va}^{\ot 2N'}+\id_{\Va}\ot B')\circ I\circ (c\cdot\lcoev_{\Va}^{\ot 2N}\ot\id_{\Va}+B\ot\id_{\Va})\\
    &=cc'\cdot (\id_{\Va}\ot \lev_{\Va}^{\ot 2N'})I(\lcoev^{\ot 2N}\ot\id_{\Va})+ c'(\id_{\Va}\ot \lev_{\Va}^{\ot 2N'})I(B\ot\id_{\Va})\\
    &+c(\id_{\Va}\ot B')I(\lcoev_{\Va}^{\ot 2N}\ot \id_{\Va})+(\id_{\Va}\ot B')I(B\ot \id_{\Va}) & \\
    &=cc'\cdot \id_{\Va}+c'(\id_{\Va}\ot \lev_{\Va}^{\ot 2N})(B\ot\id_{\Va})+c(\id_{\Va}\ot B')(\lcoev_{\Va}^{\ot 2N'}\ot \id_{\Va}) & \\
    &+(\id_{\Va}\ot B')I(B\ot \id_{\Va}) & \\
    &=cc'\cdot \id_{\Va}+c'\cdot\cl(B)+c\cdot \cl(B')+(\id_{\Va}\ot B')I(B\ot \id_{\Va}) &
\end{align*}
By the bounds on the coefficients of $\cl(B),\cl(B')$ of Lemma \ref{lemma: B has smaller deg after closing} and the bounds on $B,B'$ of Proposition \ref{prop: prop surface and coevaluations}, it follows that the top term of $F(L_o,\Va)$ is given by $cc'\cdot \id_{\Va}$. But the top term of this is precisely the product of the top terms of the ADO invariants of $L_1$ and $L_2$, therefore proving the theorem.

\subsection{Sharpness of the ADO genus bound}\label{subs: sharpness of ADO bound}  It is an interesting question to what extent the genus bound for $\ADOgr$ of \cite{LNV:genus-unrolled} is sharp. By Corollary \ref{corollary: fibred}, it is sharp on fibred links. Theorem \ref{theorem: Murasugi sum multiplicativity of ADO} has the following immediate corollary:

\begin{Corollary}
If $L$ is a Murasugi-sum of links $L_1,L_2$ along connected minimal genus Seifert surfaces, then the genus bound for $\ADOgr(L)$ of \cite{LNV:genus-unrolled} is sharp iff it is sharp for both $\ADOgr(L_1)$ and  $\ADOgr(L_2)$.
\end{Corollary}

Thus, the question of sharpness of the genus bound can be studied by decomposing links as Murasugi-sums along minimal genus Seifert surfaces (whenever possible). On the other hand, Theorem \ref{theorem: top coefficient formula with integral} implies that sharpness of the genus bound in \cite{LNV:genus-unrolled} is related to non-vanishing of integrals of universal invariants of bottom tangles. In particular, universal $\uqg$-invariants of knots are relevant to this question as follows: suppose $\S$ is an annulus around a (framed) knot $K$ and let $L_K=\p\S$ (with boundary orientation). Then $Z_{\TS}$ is the universal $\uqg$-invariant of the knot $K$ itself: $Z_{\TS}=Z_K \in \uqg$. Thus $\topgr(L_K)=\frac{\ov{\mu}(Z_K)}{(q^{(2\rho,2\rho)}\ov{\mu}(v))^N}$. In particular, $\topgr(L_K)$ is non-zero (i.e. the genus bound of \cite{LNV:genus-unrolled} is sharp) if and only if $\ov{\mu}(Z_K)\neq 0$.

\bibliographystyle{amsplain}
\bibliography{z-referencesabr2}

\end{document}